\documentclass[a4paper]{article}

\usepackage{amsmath}
\usepackage{cite}
\usepackage{tikz}
\usepackage{tikzscale}
\usetikzlibrary{matrix}
\newlength\figureheight
\newlength\figurewidth
\usepackage{pgfplots}
\usepackage{amssymb}
\usepackage{mathdots}
\usepackage{multirow}
\usepackage{algorithm}
\usepackage{algpseudocode}

\usepackage{booktabs}

\usepackage{subcaption}

\usepackage{url}

\usepackage{amsthm}
\newtheorem{theorem}{Theorem}[section]
\newtheorem{lemma}[theorem]{Lemma}
\newtheorem{corollary}[theorem]{Corollary}

\newtheorem{problem}{Problem}[section]

\newtheorem{example}{Example}[section]

\def\rddots{\cdot^{\displaystyle\cdot^{\displaystyle\cdot}}}

\title{On generating Sobolev orthogonal polynomials}

\date{}
\author{Niel Van Buggenhout\footnotemark[1]}

\begin{document}
	\maketitle
	
	\renewcommand{\thefootnote}{\fnsymbol{footnote}}
	\footnotetext[1]{Charles University, Sokolovská 83, 186 75 Praha 8, Czech Republic. (email: buggenhout@karlin.mff.cuni.cz )}
	
	\begin{abstract}
		Sobolev orthogonal polynomials are polynomials orthogonal with respect to a Sobolev inner product, an inner product in which derivatives of the polynomials appear.
		They satisfy a long recurrence relation that can be represented by a Hessenberg matrix.
		The problem of generating a finite sequence of Sobolev orthogonal polynomials can be reformulated as a matrix problem.
		That is, a Hessenberg inverse eigenvalue problem, where the Hessenberg matrix of recurrences is generated from certain known spectral information.
		Via the connection to Krylov subspaces we show that the required spectral information is the Jordan matrix containing the eigenvalues of the Hessenberg matrix and the normalized first entries of its eigenvectors.
		Using a suitable quadrature rule the Sobolev inner product is discretized and the resulting quadrature nodes form the Jordan matrix and associated quadrature weights are the first entries of the eigenvectors.
		We propose two new numerical procedures to compute Sobolev orthonormal polynomials based on solving the equivalent Hessenberg inverse eigenvalue problem.
	\end{abstract}
		\section{Introduction}
	Sobolev orthogonal polynomials are polynomials orthogonal with respect to a Sobolev inner product, i.e., an inner product in which derivatives of the polynomials appear.
	These polynomials are well studied in terms of their analytic properties, such as their asymptotic behavior and the location of their zeros, see, e.g., \cite{MF01,MaXu15,LLPCPI01}.
	The connection to Krylov subspace methods, and to numerical linear algebra in general, is largely unexplored.
	
	If the inner product involves only the polynomials themselves and not their derivatives, then there is a fruitful connection between the field of approximation theory (least squares approximation), the field of numerical linear algebra (Krylov subspaces methods) and the field of classical analysis (orthogonal polynomials).
	The problem of least squares function approximation in a given set of nodes can be best formulated in a basis that consists of specific orthogonal polynomials, i.e., polynomials orthogonal with respect to an inner product determined by the given nodes \cite{Fo57}.
	The method of choice for the numerical generation of these orthogonal polynomials, if they are not classical polynomials \cite{Sz75}, is based on reformulating the problem as a matrix problem.
	This reformulation relies on the connection between these orthogonal polynomials and Krylov subspaces generated by the diagonal matrix with the given nodes as its entries.
	The resulting problem is a Hessenberg or Jacobi inverse eigenvalue problem where the spectral information is a diagonal matrix whose entries are the given nodes \cite{Ga04book,GoMe09,GrHa84,GoWe69}.
	It can be solved by the Lanczos or Arnoldi iteration \cite{BrNaTr21,Re91,dBGo78}, by an updating procedure \cite{GrHa84,ReAmGr91} or by a method based on the quotient difference algorithm \cite{La99}.
	Furthermore, convergence analysis of Krylov subspace methods, e.g., for eigenvalue approximation or linear system solving \cite{LiSt13,Me06}, uses results from potential theory \cite{Ku06}.
	Potential theory is closely related to the study of orthogonal polynomials \cite{StTo92}.
	
	For Sobolev orthogonal polynomials there is a connection to Hermite least squares approximation, where not only the function values are known in a given set of nodes but also the derivatives of this function are known in a (sub)set of these nodes.
	This connection has been known since the initial papers on Sobolev orthogonal polynomials \cite{Al62,Gr67,IsKoNoSS89}.
	Recently a procedure for solving the Hermite least squares problem based on Krylov subspace methods was proposed \cite{NiZhZh23}. 
	We uncover the connection of Sobolev orthogonal polynomials and Hermite least squares approximation to Krylov subspaces, where the subspace is generated by a Jordan matrix.
	This new connection allows for the development of two new algebraic procedures to generate a sequence of Sobolev orthogonal polynomials. 
	It also opens the path to the application of results on Sobolev orthogonal polynomials to analyze the convergence behavior of Krylov subspace methods applied to defective matrices and to the development of spectral solvers for differential equations based on Krylov subspace methods \cite{BeMa92,YuWaLi19}.
	
	Sobolev orthonormal polynomials are introduced in Section \ref{sec:SOPs}, where also the main problem is stated, i.e., the generation of a sequence of Sobolev orthonormal polynomials.
	The connection between Krylov subspaces and Sobolev orthonormal polynomials is proved in Section \ref{sec:matrixForm}.
	This connection allows us to formulate the Hessenberg inverse eigenvalue problem that is equivalent to the main problem.
	Section \ref{sec:procedures} proposes two new methods for the solution of the inverse eigenvalue problem, one based on the Arnoldi iteration and the other based on a procedure using plane rotations that can be traced back to Rutishauser \cite{Ru63}.
	Numerical experiments are performed in Section~\ref{sec:num}, where the two new methods are compared to the state-of-the-art methods proposed by Gautschi and Zhang \cite{GaZh95}, which are generalizations of the modified Chebyshev and the discretized Stieltjes procedure \cite{Ga04book}.

	\section{Sobolev orthonormal polynomials}\label{sec:SOPs}
	Sobolev orthonormal polynomials (SOPs) form a sequence of polynomials orthonormal with respect to a Sobolev inner product.
	The Sobolev inner product is introduced in Section \ref{sec:SobInprod}.
	Using this inner product Section \ref{sec:SOP} formally defines a sequence of Sobolev orthogonal polynomials and provides some of its properties.
	Section \ref{sec:prob} formulates the main problem handled in this paper, the generation of a sequence of Sobolev orthogonal polynomials given a discretized inner product.
	Examples of Sobolev inner products and how to discretize them are given in Section \ref{sec:examples}.
	Such discretizations are needed for the formulation of the matrix problem in Section \ref{sec:matrixForm}.
	\subsection{Sobolev inner product}\label{sec:SobInprod}
	A \emph{Sobolev inner product} is defined on the space of polynomials $\mathcal{P}$ using $s+1$ finite positive Borel measures $\{\mu_r\}_{r=0}^s$ as:
	\begin{equation}\label{eq:SobInprod}
		(p,q)_S := \sum_{r=0}^{s} \int_{\Omega} p^{(r)}(z) \overline{q^{(r)}(z)} d \mu_r(z),
	\end{equation}
	where the support of $\mu_r$, denoted by $\textrm{supp}(\mu_r),$ is a compact subset of the complex plane $\mathbb{C}$.
	A Sobolev inner product of the form \eqref{eq:SobInprod} is said to be \textit{diagonal}, since, only products between polynomials of the same order of derivative appear.
	For non-diagonal Sobolev inner products we refer to the survey papers \cite{MF01,MaXu15}.
	Furthermore, we assume that the inner product is \textit{sequentially dominated} \cite{LLPC99}, that is, $\textrm{supp}(\mu_r)\subset \textrm{supp}(\mu_{r-1})$ and $d\mu_r = f_{r-1} d_{\mu_{r-1}}$, with $f_{r-1}\in L_{\infty}(\mu_{r-1})$, $r=1,\dots, s$.\
	
	A \textit{discretized inner product}\footnote[1]{The term {discrete inner product} is avoided because it is already used for a specific type of Sobolev inner product, e.g., see \cite{MaOsRo02} and Section \ref{sec:discreteSobInProd}.} will refer to an inner product that comes from the discretization of a Sobolev inner product.
	Discretization results in a finite sum over polynomials (and their derivatives) evaluated in a (sub)set of the given nodes $\{z_j\}_{j=1}^n$ weighted by real, nonnegative values $\{\vert \beta_{j,i}\vert^2\}_{i=0}^{k_j}$, where $k_j\leq s$ is the highest order derivative associated with the $j$th node.
	Hence, a discretized diagonal Sobolev inner product is an inner product of the form
	\begin{equation}\label{eq:discrSobInprod}
		\langle p,q \rangle_S  = \sum_{j=1}^{n} \begin{bmatrix}
			p(z_j) &  \dots & p^{(k_j)}(z_j) \end{bmatrix}\begin{bmatrix}
			\vert \beta_{j,0}\vert^2 \\
			& \ddots \\
			& & \vert \beta_{j,k_j}\vert^2
		\end{bmatrix}
		\begin{bmatrix}
			\overline{q(z_j)}\\
			\vdots \\
			\overline{q^{(k_j)}(z_j)}
		\end{bmatrix}.
	\end{equation}
	Such an inner product is a sequentially dominated inner product if $\beta_{j,i}\neq 0$ implies that $\beta_{j,i-1},\beta_{j,i-2},\dots,\beta_{j,0} \neq 0$, i.e., if the weight for the $i$th order derivative is nonzero then all weights corresponding to lower order derivatives and to the same node are also nonzero.
	In the remainder of this text we assume that $\beta_{j,k_j}\neq 0$ for $j=1,2,\dots,n$ and that $\langle .,. \rangle_S$ is sequentially dominated.\\
	An early reference for discretized inner products related to measures in the complex plane is \cite{CaMa90}, where the connection between the Gram matrix for Sobolev inner products and Sobolev orthogonal polynomials is discussed, a connection that is not discussed in this paper.
	
	\subsection{Sobolev orthonormal polynomials}\label{sec:SOP}
	A sequence of polynomials $\{p_k\}_{k\geq 0}$ is called a sequence of \textit{Sobolev orthonormal polynomials} (SOPs) with respect to $(.,.)_S$ if
	\begin{equation*}
		p_k \in\mathcal{P}_{k}\backslash \mathcal{P}_{k-1} \quad\text{ and }\quad
		(p_k, p_\ell)_S = \begin{cases}
			0,\quad \text{if } k\neq \ell\\
			1,\quad \text{if } k= \ell
		\end{cases}.
	\end{equation*}
	Note that this definition also includes the discretized Sobolev inner products, we can choose $(.,.)_S = \langle .,.\rangle_S$ of the form \eqref{eq:discrSobInprod}.
	A sequence $\{p_k\}_{k\geq 0}$ of SOPs satisfies the recurrence relation 
	\begin{equation}\label{eq:RR_Hess}
		z\begin{bmatrix}
			p_0(z) & \dots & p_{k-1}(z)
		\end{bmatrix} = \begin{bmatrix}
			p_0(z) &  \dots & p_{k-1}(z)
		\end{bmatrix} H_{k} + h_{k+1,k} p_{k}(z) e_k^\top,
	\end{equation}
	where $h_{k+1,k} >0$ and $H_{k}\in\mathbb{C}^{k\times k}$ is a Hessenberg matrix.
	A Hessenberg matrix $H$ is a matrix with only zero entries below its first subdiagonal, i.e., $h_{i,j} = 0$ for $i-1>j$.
	Each root of the Sobolev orthonormal polynomial $p_{k}(z)$ is an eigenvalue of the Hessenberg matrix $H_{k}\in \mathbb{C}^{k\times k}$ \cite{LLPC99}, the converse also holds since the sequence of SOPs forms a triangle family of polynomials \cite{SaSm88}.
	
	For the discretized inner product $\langle .,.\rangle_S$ and $m=\sum_{j=1}^n(k_j+1)$ we get $h_{m+1,m} = 0$ in the recurrence relation \eqref{eq:RR_Hess} and, therefore, it breaks down:
	\begin{equation*}
		z \begin{bmatrix}
			p_0(z) & \dots & p_{m-1}(z)
		\end{bmatrix} = \begin{bmatrix}
			p_0(z) & \dots & p_{m-1}(z)
		\end{bmatrix} H_m.
	\end{equation*} 
	Moreover, the polynomial $p_m(z)$ is known up to normalization, i.e., $p_m(z) = \prod_{j=0}^{n}(z-z_j)^{k_j+1}$. This polynomial cannot be normalized with $\langle .,. \rangle_S$ since it vanishes for this inner product, i.e., $\langle p_m,f\rangle_S=0$ for any polynomial $f\in\mathcal{P}$.
	
	\subsection{The problem of generating SOPs}\label{sec:prob}
	The problem considered in this paper is the following.
	\begin{problem}[Generate SOPs]\label{prob:generateSOPs}
		For a given sequentially dominated diagonal Sobolev inner product $\left(.,.\right)_S$, compute the sequence of Sobolev orthonormal polynomials $\{p_k\}_{k=0}^{m-1}$, i.e.,
		\begin{equation*}
			p_k \in\mathcal{P}_{k}\backslash \mathcal{P}_{k-1} \quad\text{ and }\quad
			(p_k, p_\ell)_S = \begin{cases}
				0,\quad \text{if } k\neq \ell\\
				1,\quad \text{if } k= \ell
			\end{cases}.
		\end{equation*}
	\end{problem}
	The procedures proposed in this paper rely on first discretizing the inner product $\left(.,.\right)_S$, which results in an inner product $\langle.,.\rangle_S$ of the form \eqref{eq:discrSobInprod} and, second,
applying techniques from numerical linear algebra to generate a finite sequence of Sobolev orthonormal polynomials for $\langle.,.\rangle_S$.
	For an adequate choice of discretization, these polynomials will be orthonormal with respect to both the discretized inner product, $\langle .,.\rangle_{S}$, and the continuous one, $\left(.,.\right)_S$.
	For example, if a Gaussian quadrature rule with degree of exactness $2m-1$ is used to obtain $\langle.,.\rangle_{S}$, then it is expected that we can generate the first $m$ SOPs $\{p_0,p_1,\dots,p_{m-1}\}$ for $\left(.,.\right)_S$ accurately.
	We assume that an appropriate discretization $\langle .,.\rangle_{S}$ of $(.,.)_S$ is available, then the problem can be formulated in terms of this discretized inner product.
	\begin{problem}[Generate SOPs]\label{prob:generateDiscreteSOPs}
		Given a sequentially dominated diagonal discretized Sobolev inner product $\langle.,.\rangle_S$, as in \eqref{eq:discrSobInprod}, with $n$ nodes $z_j$ of multiplicity $(k_j+1)$ and corresponding weights $\{\beta_{j,i}\}_{i=0}^{k_j}$.
		Let $m = \sum_{j=1}^n (k_j+1)$.
		Compute a sequence of Sobolev orthonormal polynomials $\{p_k\}_{k=0}^{m-1}$ for $\langle.,.\rangle_S$, i.e.,
$p_k\in\mathcal{P}_k\backslash\mathcal{P}_{k-1}$ and
		\begin{equation*}
			\langle p_k,p_\ell\rangle_S =\begin{cases}
				0,\quad \text{if } k\neq \ell,\\
				1,\quad \text{if } k= \ell.
			\end{cases}
		\end{equation*}
	\end{problem}
	This formulation can be immediately reformulated as a matrix problem, see Section~\ref{sec:IEP}.
	
	\subsection{Examples}\label{sec:examples}
	The discretization of Sobolev inner products using Gauss-type quadrature rules is illustrated, see also \cite[Chapter 2.2]{Ga04book}.
	An $n$-point Gaussian quadrature rule for the measure $d\mu$ on $\Omega$ consists of a set of $n$ nodes $\{z_j\}_{j=1}^n$ and weights $\{\beta_j\}_{j=1}^n$ such that
	\begin{equation*}
		\int_{\Omega} f(z)d \mu(z) = \sum_{j=1}^{n} \vert\beta_j\vert^2 f(z_j) + R_n(f),
	\end{equation*}
	where $R_n(f) = 0$ for any $f\in\mathcal{P}_{2n-1}$.
	That is, the quadrature rule is exact for polynomials up to degree $2n-1$.
	
	\subsubsection{Same measures $\mu_r= \gamma_r \mu$}
	Consider a Sobolev inner product where all measures are the same up to a multiplicative constant, $\mu_r = \gamma_r\mu$ for $r=0,1,\dots,s$ and $\gamma_r>0$, 
	\begin{equation*}
		(p,q)_S = \sum_{r=0}^{s} \gamma_r \int_{\Omega} p^{(r)}(z) \overline{q^{(r)}(z)} d \mu(z).
	\end{equation*}
	These Sobolev inner products occur in the earliest study of Sobolev orthogonal polynomials \cite{Al62,Gr67,MaPePi96}.
	Using an $n$-point Gauss quadrature rule with nodes $\{z_j\}_{j=1}^n$ and weights $\{\vert\beta_j\vert^2\}_{j=1}^n$ for the integral, results in the discretized inner product
	\begin{equation*}
		\langle p,q\rangle_S =  \sum_{j=1}^{n} \vert \beta_j \vert^2 \left(\sum_{r=0}^{k_{j}} \gamma_r p^{(r)}(z_j) \overline{q^{(r)}(z_j)}\right).
	\end{equation*}
	This discretization is exact up to degree $2m-1$, for $m=\sum_{j=1}^n (k_j+1)$, in the sense that $(p,q)_S = \langle p,q\rangle_S$ for all $p\in\mathcal{P}_{m-1}$ and $q\in\mathcal{P}_m$.
	Thus, it suffices to generate the $m$ first SOPs $\{p_k\}_{k=0}^{m-1}$ for $(.,.)_S$.
	Althammer polynomials \cite{Al62} are orthogonal with respect to a Sobolev inner product of this form, with $\mu$ the Legendre measure and $s=1$:
	\begin{equation}\label{eq:AltPoly}
		\langle p,q\rangle_{S} = \int_{-1}^{1} p(x) q(x) dx + \gamma \int_{-1}^{1} p'(x) q'(x) dx, \quad \gamma >0.
	\end{equation}
	Discretization with an $n$-point Gauss-Legendre quadrature rule, with nodes $\{x_j^\textrm{Le}\}_{j=1}^n$ and weights $\{\vert\beta_j^\textrm{Le}\vert^2 \}_{j=1}^n$, results in the discretized inner product
	\begin{equation*}
		\langle p,q\rangle_{S} = \sum_{j=1}^n \vert \beta_j^\textrm{Le}\vert^2 p(x_j^\textrm{Le}) q(x_j^\textrm{Le}) + \gamma \sum_{j=1}^n \vert \beta_j^\textrm{Le}\vert^2 p'(x_j^\textrm{Le}) q'(x_j^\textrm{Le}).
	\end{equation*}
	Marcell\'an and co-authors \cite{MaPePi96} studied Laguerre-Sobolev polynomials, which are orthogonal with respect to the inner product
	\begin{equation}\label{eq:LaguerreInprod}
		(p,q)_S = \int_{0}^{+\infty} p(x) q(x) x^\alpha \exp(-x) dx  + \gamma \int_{0}^{+\infty} p'(x) q'(x) x^\alpha \exp(-x) dx.
	\end{equation}
	An $n$-point Gauss-Laguerre quadrature rule, with nodes $\{x_j^\textrm{La}\}_{j=1}^n$ and weights by $\{\vert\beta_j^\textrm{La}\vert^2 \}_{j=1}^n$, results in a suitable discretized inner product
	\begin{equation}
\label{eq:LaguerreInprod_disc}
		\langle p,q\rangle_{S} = \sum_{j=1}^n \vert \beta_j^\textrm{La}\vert^2 p(x_j^\textrm{La}) q(x_j^\textrm{La}) + \gamma \sum_{j=1}^n \vert \beta_j^\textrm{La}\vert^2 p'(x_j^\textrm{La}) q'(x_j^\textrm{La}).
	\end{equation}
	
	\subsubsection{Discrete Sobolev inner products}\label{sec:discreteSobInProd}
	Discrete Sobolev inner products are composed of an integral on the functions themselves and a weighted sum over a finite set of nodes $\{c_j\}_{j=1}^J$ of the derivatives of the functions up to the order $\ell_j$ for the node $c_j$, see, e.g., \cite{MaOsRo02},
	\begin{equation*}
		(p,q)_S = \int_{\Omega}p(z) \overline{q(z)} d\mu(z) + \sum_{j=1}^{J} \sum_{r=0}^{\ell_j} \gamma_k^{(r)} p^{(r)}(q_j) q^{(r)}(c_j).
	\end{equation*}
	If the nodes $c_j\notin \Omega$, then we can use the $n$-point Gauss quadrature rule $\sum_{j=1}^{n} \vert \beta_j \vert^2 p(z_j) q(z_j) \approx \int_{\Omega}p(z) \overline{q(z)} d\mu(z)$ to obtain a suitable discretized inner product for $(.,.)_S$,
	\begin{equation}\label{eq:discreteSobInprod_disc}
		\langle p,q\rangle_S = \sum_{j=1}^{n} \vert \beta_j \vert^2 p(z_j) q(z_j) + \sum_{j=1}^{J} \sum_{r=0}^{\ell_j} \gamma_k^{(r)} p^{(r)}(c_j) q^{(r)}(c_j).
	\end{equation}
	The discrete Laguerre-Sobolev inner product is of the above form, that is,
	\begin{equation}\label{eq:LagSob}
		(p,q)_S = \int_{0}^{+\infty} p(x) q(x) x^\alpha \exp(-x) dx  + M p(c)q(c) + N p'(c)q'(c),
	\end{equation}
	with $\alpha >-1$, real numbers $M,N>0$ and $c=-1$.
	There exists a higher order five-term recurrence relation for the associated SOPs \cite{HeHuLaMa22}. 
	Next, we consider an example where the nodes are endpoints of the interval $\Omega = \left[a,b\right]\subset\mathbb{R}$,
	\begin{equation*}
		(p,q)_S = \int_{\Omega}p(x) q(x) d\mu(x) + \sum_{r=0}^{\ell_1} \gamma_1^{(r)} p^{(r)}(a) q^{(r)}(a) + \sum_{r=0}^{\ell_2} \gamma_2^{(r)} p^{(r)}(b) q^{(r)}(b).
	\end{equation*}
	In order to obtain a discretized inner product that is sequentially dominated, a Gauss-Radau or Gauss-Lobatto rule must be used.
	Consider, for example,
	\begin{equation*}
		(p,q)_S = \int_{-1}^{1} p(x)q(x) d\mu(x) + \gamma p'(1)q'(1).
	\end{equation*}

	For a sequentially dominated discretized inner product, the quadrature rule for the integral must have a term $p(1)q(1)$.
	Therefore, a Gauss-Radau rule with a node fixed at $x=1$ is used.
	The resulting nodes $\{x_j^R\}_{j=0}^n$, with $x_0^R = 1$, and weights $\{\vert \beta_j\vert^2\}_{j=0}^n$ lead to the discretized inner product
	\begin{align}
		\int_{1}^{1} p(x)q(x) d\mu(x) + \gamma p'(1)q'(1) \approx & \sum_{j=1}^{n} \vert \beta_j^R\vert^2 p(x_j^R) q(x_j^R) \nonumber\\ 
		&+ \vert \beta_0^R \vert^2 p(1) q(1) + \gamma p'(1)q'(1).
\label{eq:GRplusa}
	\end{align}

	\section{Matrix formulation}\label{sec:matrixForm}
	The problem of generating Sobolev orthonormal polynomials, Problem \ref{prob:generateDiscreteSOPs}, is reformulated as a matrix problem.
	This is a Hessenberg inverse eigenvalue problem, where a Hessenberg matrix must be constructed from spectral information.
	By using the connection between Sobolev orthogonal polynomials and Krylov subspaces we show that the required spectral information is the Jordan matrix containing the eigenvalues of the Hessenberg matrix and the normalized first entries of its eigenvectors.
	Krylov subspaces are introduced in Section~\ref{sec:KrylovSubspaces}.
	In Section \ref{sec:inprods} we prove that there is an equivalence between certain Sobolev inner products on the space of polynomials and the Euclidean inner product on Krylov subspaces.
	This equivalence allows us to formulate the Hessenberg inverse eigenvalue problem in Section \ref{sec:IEP}.
	The Jordan matrix and first entries of the eigenvectors are obtained from the discretized inner product, Section~\ref{sec:examples_MatrixForm} illustrates for some examples how to obtain these.
	
	\subsection{Krylov subspaces}\label{sec:KrylovSubspaces}
	A \emph{Krylov subspace} of dimension $k\leq m$ is defined for a matrix $A\in\mathbb{C}^{m\times m}$ and starting vector $v\in\mathbb{C}^m$ as 
	\begin{equation*}
		\mathcal{K}_k(A,v) = \textrm{span}\{v,Av,A^2 v,\dots,A^{k-1}v \}.
	\end{equation*}
	Krylov subspaces are nested, i.e., $\mathcal{K}_{k-1}(A,v)\subseteq \mathcal{K}_k(A,v)$.
	A basis $\begin{bmatrix}
		q_1 & \cdots & q_k
	\end{bmatrix}=Q_k\in\mathbb{C}^{m\times k}$ for $\mathcal{K}_k(A,v)$ is a \textit{nested orthonormal basis} if it satisfies  
	\begin{equation*}
		Q_k^H Q_k = I \quad \text{ and }\quad \textrm{span}\{q_1,\dots, q_i\} = \mathcal{K}_i(A,v), \quad i=1,2,\dots, k.
	\end{equation*}
	A nested orthonormal basis $Q_k$ for $\mathcal{K}_k(A,v)$ satisfies a recurrence relation of the form
	\begin{equation}\label{eq:RR_Krylov}
		A Q_k = Q_k H_k + h_{k+1,k} q_{k+1}e_k^\top,
	\end{equation}
	where $h_{k+1,k} >0$, $H_{k}\in\mathbb{C}^{k\times k}$ is a Hessenberg matrix and $q_{k+1}\in\mathbb{C}^m$ satisfies the orthogonality condition $q_{k+1}^H Q_k = \begin{bmatrix}
		0 & \cdots & 0
	\end{bmatrix}$.
	
	When $\mathcal{K}_{k}(A,v)$ is an invariant subspace for $A$, i.e., $A\mathcal{K}_{k}(A,v)\subseteq \mathcal{K}_{k}(A,v)$, then $h_{k+1,k} = 0$ and the recurrence relation breaks down.
	For the matrix-vector pairs discussed in this paper, i.e., those related to a discretized inner product $\langle .,. \rangle_{S}$, a breakdown cannot occur for $k<m$ and thus, it occurs for $k=m$.
	An important observation is that any vector $y\in\mathcal{K}_k(A,v)$ can be written as
	\begin{equation*}
		y = \pi_{k-1}(A)v, \quad \text{where }\pi_{k-1}\in\mathcal{P}_{k-1}.
	\end{equation*}
	That is, any vector in a Krylov subspace can be written as a polynomial evaluated in the matrix $A$ and multiplied by the starting vector $v$.
	
	\subsection{Equivalence between inner products}\label{sec:inprods}
	For a given sequentially dominated discretized diagonal Sobolev inner product~\eqref{eq:discrSobInprod}, it is possible to choose a matrix $A=Z$ and starting vector $v=w$ such that it is equivalent to the Euclidean inner product on $\mathcal{K}_m(Z,w)$.
	The matrix $Z$ is a Jordan matrix, i.e., a block diagonal matrix,
	\begin{equation*}
		Z = \begin{bmatrix}
			J_{1,k_1} \\
			& J_{2,k_2}\\
			& & \ddots \\
			& & & J_{n,k_n}
		\end{bmatrix}= J_{1,k_1}\oplus J_{2,k_2}\oplus \dots \oplus J_{n,k_n},
	\end{equation*}
	composed of Jordan blocks $J_{j,k_j}$ with distinct eigenvalues, $z_i\neq z_j$ if $i\neq j$,
	\begin{equation}\label{eq:JordanBlock}
		J_{j,k_j} = J_{j,k_j}(\alpha_1^{(j)},\dots,\alpha_{k_j}^{(j)}) = \begin{bmatrix}
			z_j & \alpha_{k_j}^{(j)}\\
			& z_j & \ddots\\
			& & \ddots & \alpha_1^{(j)}\\
			& & & z_j
		\end{bmatrix}\in\mathbb{C}^{(k_j+1)\times (k_j+1)}.
	\end{equation}
	The starting vector, which we will also call weight vector, $w$ has the form
	\begin{equation}\label{eq:weightBlock}
		w = \begin{bmatrix}
			\beta_1 e_{k_1+1}\\
			\beta_2 e_{k_2+1}\\
			\vdots \\
			\beta_n e_{k_n+1}
		\end{bmatrix}, \quad \text{with } e_{{k_j+1}} := \begin{bmatrix}
			0\\
			\vdots\\
			0\\
			1
		\end{bmatrix}\in\mathbb{R}^{(k_j+1)}.
	\end{equation}
	In Theorem \ref{theorem:KrylovInprod} we state this equivalence formally, first, we state a property of matrix functions of a Jordan block and associated weight vector.
	\begin{lemma}\label{lemma:colfJk}
		The last column of $f(J_{j,k_j})$, with $J_{j,k_j} := J_{j,k_j}(\alpha_1^{(j)},\dots,\alpha_{k_j}^{(j)})$ as in Equation \ref{eq:JordanBlock}, is
		\begin{equation*}
			f(J_{j,k_j}) e_{k_j} = \begin{bmatrix}
				\frac{\prod_{i=1}^{k_j}\alpha_i^{(j)}}{k_j!} f^{(k_j)}(z_j)\\
				\frac{\prod_{i=1}^{k_j-1}\alpha_i^{(j)}}{(k_j-1)!} f^{(k_j-1)}(z_j)\\
				\vdots\\
				\frac{\alpha_1^{(j)} \alpha_2^{(j)}}{2!} f''(z_j)\\
				\alpha_1^{(j)} f'(z_j)\\
				f(z_j)
			\end{bmatrix}.
		\end{equation*}
	\end{lemma}
	\begin{proof}
		Plug $J_{j,k_j}$ into the Taylor series expansion of $f(J_{j,k_j})$.
	\end{proof}
	
	\begin{theorem}\label{theorem:KrylovInprod}
		Consider $Z\in \mathbb{C}^{m\times m}$ and $w\in\mathbb{C}^{m}$, as in Equations \ref{eq:JordanBlock} and \ref{eq:weightBlock}, respectively.
		And consider two vectors $x\in\mathcal{K}_k(Z,w)$ and $y\in\mathcal{K}_{\ell}(Z,w)$, with $k,\ell\leq m$, then
		\begin{equation}\label{eq:KrylInprod}
			y^H x = \langle p,q\rangle_S =\sum_{j = 1}^{n} \vert \beta_j\vert^2 \left(\sum_{r=0}^{k_j} \left\vert \frac{\prod_{i=1}^{r} \alpha_i^{(j)}}{r!}\right\vert^2 \overline{q^{(r)}(z_j)} p^{(r)}(z_j)\right)
		\end{equation}
		for $p\in\mathcal{P}_k$ and $q\in\mathcal{P}_\ell$ satisfying $p(Z)w=x$ and $q(Z)w=y$.
	\end{theorem}
	\begin{proof}
		A vector in a Krylov subspace $x\in \mathcal{K}_k(Z,w)$ of dimension $k$ can be written as $x = p(Z)w$ for some polynomial $p\in\mathcal{P}_k$.
		Similarly for $y\in \mathcal{K}_{\ell}(Z,w)$ we have $y=q(Z)w$ for $q\in\mathcal{P}_{\ell}$.
		Therefore, by the definition of a matrix function \cite[Chapter 1]{Hi08}, the Euclidean inner product can be written as
		\begin{align*}
			y^H x &= w^H \left(q(Z)\right)^H p(Z)w \\
			&= \left(
			q\left(J_{1,k_1}\oplus \dots \oplus J_{n,k_n}\right)
			\begin{bmatrix}
				\beta_1 e_{k_1+1}\\
				\vdots\\
				\beta_n e_{k_n+1}
			\end{bmatrix}\right)^H
			p\left(J_{1,k_1}\oplus \dots \oplus J_{n,k_n}\right) \begin{bmatrix}
				\beta_1 e_{k_1+1}\\
				\vdots\\
				\beta_n e_{k_n+1}
			\end{bmatrix}.
		\end{align*}
		Consider a single block, say block $j$, and apply Lemma \ref{lemma:colfJk}:
		\begin{align*}
			\left( q(J_{j,k_j}) \beta_j e_{k_j+1} \right)^H p(J_{j,k_j}) \beta_j e_{k_j+1} &= \vert \beta_j\vert^2 \left( q(J_{j,k_j}) e_{k_j+1} \right)^H p(J_{j,k_j}) e_{k_j+1}\\
			&= \vert \beta_j\vert^2 \sum_{r=0}^{k_j} \left\vert 	\frac{\prod_{i=1}^{r} \alpha_i^{(j)}}{r!} \right\vert^2 \overline{q^{(r)}(z_j)} p^{(r)}(z_j).
		\end{align*}
		Summing over all blocks $j=1,2,\dots,n$, proves the statement.
	\end{proof}
	Theorem \ref{theorem:KrylovInprod} implies that if vectors $x,y\in \mathcal{K}_m(Z,w)$ are orthogonal for the Euclidean inner product, i.e., $y^Hx = 0$, then the polynomials $p,q\in\mathcal{P}$ satisfying $p(Z)w=x$ and $q(Z)w=y$ are orthogonal with respect to the Sobolev inner product in Equation \ref{eq:KrylInprod}.
	
	\subsection{Hessenberg inverse eigenvalue problem}\label{sec:IEP}
	From Theorem \ref{theorem:KrylovInprod} and the Gram-Schmidt orthogonalization process it follows that the Hessenberg matrix appearing in the recurrence relations \eqref{eq:RR_Hess} and \eqref{eq:RR_Krylov} is the same matrix.
	\begin{corollary}\label{cor:SOP_Qk_equivalence}
		Consider a nested orthonormal basis $Q_k\in\mathbb{C}^{m\times k}$ for the Krylov subspace $\mathcal{K}_{k}(Z,w)$, with $k> m$ and $Z\in\mathbb{C}^{m\times m}$, $w\in\mathbb{C}^m$ as in \eqref{eq:JordanBlock} and \eqref{eq:weightBlock}.
		Let $H_k\in\mathbb{C}^{k\times k}$ denote the Hessenberg matrix and $h_{k+1,k}> 0$ such that
		\begin{equation*}
			Z Q_k = Q_k H_k + h_{k+1,k}q_{k+1}e_k^\top.
		\end{equation*}
		Then, the recurrence relation for the sequence of SOPs $\{p_\ell\}_{\ell=0}^{k}$, orthonormal with respect to $\langle .,.\rangle_{S}$ \eqref{eq:KrylInprod}, is given by
		\begin{equation*}
			z \begin{bmatrix}
				p_0(z) & \dots & p_{k-1}(z) 
			\end{bmatrix} = \begin{bmatrix}
				p_0(z) & \dots & p_{k-1}(z) 
			\end{bmatrix} H_k + h_{k+1,k} p_k(z).
		\end{equation*}
		The converse also holds.
	\end{corollary}
	For $k=m$, the recurrence relation \eqref{eq:RR_Krylov} breaks down, i.e.,
	\begin{equation*}
		Z Q_m = Q_m H_m,
	\end{equation*}
	with $Q_m e_1 = w/\Vert w\Vert_2$ and $Q_m^H Q_m = I$.
	Thanks to the unitarity of $Q_m$ we have
	\begin{equation*}
		Q_m^H Z Q_m = H_m,
	\end{equation*}
	which is the Jordan canonical form of the Hessenberg matrix $H_m$.
	Conversely, if the first entries of the eigenvectors $Q_m e_1 = w/\Vert w\Vert_2$ and the Jordan matrix $Z$ are known, then the unitary matrix $Q_m$ and Hessenberg matrix $H_m$ can be reconstructed.
	This is called a structured inverse eigenvalue problem \cite{ChGo02} and, since $H_m$ must be of Hessenberg form, it is called a Hessenberg inverse eigenvalue problem.
	\begin{problem}[Hessenberg inverse eigenvalue problem (HIEP)]\label{prob:IEP}
		Given a matrix $Z = J_{1,k_1} \oplus J_{2,k_2} \oplus \dots \oplus J_{n,k_n}$ for Jordan blocks $J_{j,k_j}\in\mathbb{C}^{(k_j+1)\times (k_j+1)}$ \eqref{eq:JordanBlock}, with distinct nodes $z_j\in \mathbb{C}$, and a weight vector $w=\begin{bmatrix}
			\beta_1 e_{k_1+1}^\top & \dots & \beta_n e_{k_n+1}^\top 
		\end{bmatrix}^\top$. 
		Construct a Hessenberg matrix $H_m\in\mathbb{C}^{m\times m}$, with $m=\sum_{j = 1}^n(k_j+1)$, such that
		\begin{equation*}
			Q_m^H Z Q_m = H_m \quad \text{and}\quad Q_m e_1  = \frac{w}{\Vert w \Vert_2}
		\end{equation*}		
		hold for a unitary matrix $Q_m\in \mathbb{C}^{m\times m}$.
	\end{problem}
	The leading principal submatrix of size $k\times k$ of $H_m$ is the recurrence matrix $H_k$ which generates $Q_{k}\in\mathbb{C}^{m\times k}$ and $\{p_\ell\}_{\ell=0}^{k-1}$.
	The spectral information required to formulate Problem \ref{prob:IEP} for a Sobolev inner product $(.,.)_S$ are the nodes and weights that appear in the discretized inner product $\langle .,.\rangle_S$.

	\subsection{Examples}\label{sec:examples_MatrixForm}
	Using the discretized inner products in Section \ref{sec:examples} we illustrate how the Jordan matrix $Z$ and weight vector $w$ are constructed.
	
	\subsubsection{Same measures $\mu_r = \gamma_r \mu$}
	For Althammer polynomials \eqref{eq:AltPoly} and Laguerre-Sobolev polynomials \eqref{eq:LaguerreInprod} the discretized inner product is of the form
	\begin{equation*}
		\langle p,q\rangle_{S} = \sum_{j = 1}^n \vert \beta_j\vert^2 \left(q(z_j) p(z_j) +  \gamma q'(z_j) p'(z_j) \right).
	\end{equation*}
	The associated Jordan matrix $Z$ consists of Jordan blocks of size $2\times 2$,
	\begin{equation*}
		Z = \begin{bmatrix}
			z_1 & \sqrt{\gamma_1}\\
			& z_1\\
			& & z_2 & \sqrt{\gamma_{1}}\\
			& & & z_2\\
			& & & & \ddots\\
			& & & & & z_n & \sqrt{\gamma_{1}}\\
			& & & & & & z_n
		\end{bmatrix}\text{ and } w = \begin{bmatrix}
			0 \\
			\beta_1\\
			0\\
			\beta_2\\
			\vdots\\
			0\\
			\beta_n
		\end{bmatrix}.
	\end{equation*}

	\subsubsection{Discrete inner products}
	For the discretized inner product \eqref{eq:LagSob} the Jordan matrix consists of a single $2\times 2$ Jordan block and $n$ Jordan blocks of size $1\times 1$.
	The $2\times 2$ block is associated to the node $c=-1$ with the weight vector $\begin{bmatrix}
		0 & \sqrt{M}
	\end{bmatrix}^\top$ and the $1\times 1$ blocks to the Gauss-Laguerre nodes $\{x_j^{\textrm{La}}\}_{j=1}^n$ with associated Laguerre weights $\{\vert \beta_j^{\textrm{La}}\vert^2\}_{j=1}^n$,
	\begin{equation}\label{eq:discLagSobMV}
		Z = \begin{bmatrix}
			c &\frac{\sqrt{N}}{\sqrt{M}}\\
			& c\\
			& & x_1^{\textrm{La}}\\
			& & & x_2^{\textrm{La}}\\
			& & & & \ddots \\
			& & & & & x_n^{\textrm{La}}
		\end{bmatrix} \text{ and } w = \begin{bmatrix}
			0\\
			\sqrt{M}\\
			\beta_1^{\textrm{La}}\\
			\beta_2^{\textrm{La}}\\
			\vdots\\
			\beta_n^{\textrm{La}}
		\end{bmatrix}.
	\end{equation}
	
	For the inner product \eqref{eq:GRplusa} the matrix and vector have the same form
	\begin{equation*}
		Z = \begin{bmatrix}
			a &\frac{\sqrt{\gamma}}{\beta_0^R}\\
			& a\\
			& & x_1^R\\
			& & & x_2^R\\
			& & & & \ddots \\
			& & & & & x_n^R
		\end{bmatrix} \text{ and } w = \begin{bmatrix}
			0\\
			\beta_0^R\\
			\beta_1^R\\
			\beta_2^R\\
			\vdots\\
			\beta_n^R
		\end{bmatrix}.
	\end{equation*}

	However, the entries are different, since these are Gauss-Radau nodes $\{x_j^R\}_{j=0}^n$ and weights $\{\vert \beta_j^R\vert^2\}_{j=0}^n$, where $x_0^R = a$.
	
	\section{Generating Sobolev orthonormal polynomials}\label{sec:procedures}
	Thanks to the reformulation of Problem \ref{prob:generateDiscreteSOPs} into the matrix problem, Problem \ref{prob:IEP}, we can use numerical linear algebra techniques to generate SOPs.
	For similar problems this methodology has proven to be very powerful.
	For example, consider polynomials orthogonal with respect to an inner product, not involving derivatives, on the real line.
	The method of choice to generate these orthogonal polynomials is based on solving a Jacobi inverse eigenvalue problem with techniques from structured matrix theory \cite{GrHa84,Ru63,La99,Ru57}.
	Other examples are algorithms for modifying an inner product and the associated sequence of orthogonal polynomials or rational functions \cite{GrHa84,VBVBVa23} and for solving least squares problems \cite{BrNaTr21,ReAmGr91,Re91}.
	In the following sections we propose two methods to generate Sobolev orthonormal polynomials, the Arnoldi iteration from the study of Krylov subspace methods and an updating procedure based on the manipulation of structured matrices.

	\subsection{Arnoldi iteration}\label{sec:Arnoldi}
	The Arnoldi iteration \cite{Ar51}, given in Algorithm \ref{alg:Arnoldi}, is the advised way to generate an orthonormal nested basis $Q_k\in\mathbb{C}^{m\times k}$ for a Krylov subspace $\mathcal{K}_k(A,v)$ generated with a non-Hermitian matrix $A^H\neq A$.
	From Corollary \ref{cor:SOP_Qk_equivalence} it follows that for $A=Z$ and $v=w$, the resulting Hessenberg matrix $H_k\in\mathbb{C}^{k\times k}$ is the recurrence matrix for a sequence of $k$ SOPs.
	Note that for $Z$ and $w$ as in \eqref{eq:JordanBlock} and \eqref{eq:weightBlock}, the Arnoldi iteration cannot have an early breakdown.
	An implementation of the following algorithm is available online \cite{git_SOP}.
	\begin{algorithm}
		\caption{Arnoldi iteration 	\cite{Ar51}}\label{alg:Arnoldi}
		\begin{algorithmic}[1]
			\State \textbf{Input:} $A \in \mathbb{C}^{m\times m}$, $v\in\mathbb{C}^m$, integer $k\leq m$
			\State \textbf{Output:} Orthonormal matrix $Q_{k} \in \mathbb{C}^{m\times k}$, vector $q_{k+1}\in\mathbb{C}^m$, Hessenberg matrix $H_k \in\mathbb{C}^{k\times k}$ and $h_{k+1,k}\geq 0$ such that \eqref{eq:RR_Krylov} is satisfied.
			\Procedure{Arnoldi\_iteration}{$A,v,k$}
			\State $q_1 = v/ \Vert v\Vert_2$
			\For{$\ell=1,2,\dots,k$}
			\State $q_{\ell+1} = Aq_{\ell}$
			\For{$j=1,2,\dots, \ell$}\Comment{Orthogonalization}
			\State $h_{j,\ell} = q_j^H q_{\ell+1}$
			\State $q_{\ell+1} = q_{\ell+1} - h_{j,\ell} q_j$ 
			\EndFor
			\State $h_{\ell+1,\ell} = \Vert q_{\ell+1}\Vert_2$
			\If{$h_{\ell+1,\ell} = 0$}
			\State Return $Q_\ell = \begin{bmatrix}
				q_1 & \dots & q_\ell
			\end{bmatrix}$ and $H_\ell = \left[h_{i,j}\right]_{i,j=1}^{\ell}$ \Comment{Breakdown}
			\EndIf
			\State $q_{\ell+1} = q_{\ell+1}/h_{\ell+1,\ell}$ \Comment{Normalization}
			\EndFor
			\State Return $Q_k = \begin{bmatrix}
				q_1 & \dots & q_k
			\end{bmatrix}$ and $H_k = \left[h_{i,j}\right]_{i,j=1}^{k}$
			\EndProcedure\end{algorithmic}
	\end{algorithm}
	
	This algorithm is the matrix version of the discretized Stieltjes procedure for generating Sobolev orthonormal polynomials \cite{GaZh95,GrHa84}.
	
	Another Krylov subspace method suitable for non-Hermitian matrices is the biorthogonal Lanczos iteration \cite{La50}.
	The biorthogonal Lanczos iteration applied to Jordan matrices is related to formal orthogonal polynomials \cite{PoPrSt18}.

	\subsection{Updating procedure}\label{sec:up}
	Based on an idea that can be traced back to Rutishauser \cite{Ru63}, Gragg and Harrod developed a numerically stable procedure to generate a Jacobi matrix representing a sequence of orthonormal polynomials on the real line \cite{GrHa84}.
	This procedure has been generalized to several other cases, e.g., to orthogonal polynomials on the unit circle \cite{AmGrRe91b} and to orthogonal rational functions \cite{VBVBVa22}.
	
	Here we generalize it to Sobolev orthogonal polynomials.
	The Sobolev inner product of interest \eqref{eq:KrylInprod} can be split into the sum of two inner products:
	\begin{align*}
		\langle p,q\rangle_S
		&=\underbrace{\sum_{j = 1}^{n-1} \vert \beta_j\vert^2 \left(\sum_{r=0}^{k_j} \left\vert \frac{\prod_{i=1}^{k_j-r} \alpha_i^{(j)}}{(k_j-r)!}\right\vert^2 \overline{q^{(k_j-r)}(z_j)} p^{(k_j-r)}(z_j)\right)}_{=:\langle p,q\rangle_{\hat{S}}} \\
		&\qquad\qquad + \underbrace{\vert \beta_n\vert^2 \left(\sum_{r=0}^{k_n} \left\vert \frac{\prod_{i=1}^{k_n-r} \alpha_i^{(n)}}{(k_n-r)!}\right\vert^2 \overline{q^{(k_n-r)}(z_n)} p^{(k_n-r)}(z_n)\right)}_{=:\langle p,q\rangle_{1}}.
	\end{align*}
	This corresponds to splitting the associated matrix and vector as
	\begin{align*}
		Z &= \left[
		\begin{array}{ccc|c}
			J_{1,k_1} & &  \\
			& \ddots& &  \\
			& & J_{n-1,k_{n-1}}\\ \hline
			& & & J_{n,k_n}
		\end{array}
		\right] =: \left[
		\begin{array}{c|c}
			\hat{Z}&  \\ \hline
			& J_{n,k_n}
		\end{array}
		\right] \text{ and }\\
		w &= \left[
		\begin{array}{c}
			\beta_1 e_{k_1+1}\\
			\vdots\\
			\beta_{n-1}e_{k_{n-1}+1}\\ \hline
			\beta_n e_{k_n+1}
		\end{array}
		\right] =:  \left[
		\begin{array}{c}
			\vline \\
			\hat{w}\\
			\vline\\ \hline
			\beta_n e_{k_n+1}
		\end{array}
		\right].
	\end{align*}
	The sequence of SOPs for $\langle .,. \rangle_{\hat{S}}$ is found by solving a HIEP with $\hat{Z}$ and $\hat{w}$ and for $\langle .,. \rangle_{1}$ by solving a HIEP with $J_{n,k_n}$ and $\beta_n e_{k_n+1}$.
	
	The key idea of updating procedures is to start from these two solutions and efficiently construct the SOPs for $\langle .,.\rangle_{S} = \langle .,. \rangle_{\hat{S}} + \langle .,.\rangle_{1}$.
	Such a procedure is composed of three steps, which are discussed in the following three sections.
	First the solutions to these two smaller problems, of size $\sum_{j=1}^{n-1} (k_j+1)$ and $(k_n+1)$, are embedded in matrices of size $\sum_{j=1}^{n} (k_j+1)$.
	The resulting recurrence matrix has the nodes of $\langle .,.\rangle_S$ as its eigenvalues.
	Second, the first column of the corresponding basis matrix is changed such that it is equal to the normalized weights of $\langle .,.\rangle_S$.
	This is achieved by multiplying with a unitary matrix.
	Since this changes the basis matrix, the recurrence matrix must be changed accordingly, i.e., by applying a similarity transformation with that unitary matrix.
	This similarity transformation alters the structure of the recurrence matrix, it is no longer a Hessenberg matrix.
	Third, the structure of the recurrence matrix is restored to Hessenberg form by means of unitary similarity transformations.
	
	Whereas the literature only considers inner products $\langle .,.\rangle_{1}$ consisting of a single term, which has a scalar solution, a $1\times 1$ matrix. 
	Here we consider inner products with $k_n+1$ terms and thus, these have a $(k_n+1)\times (k_n+1)$ matrix solution.
	The idea of updating with larger matrices first appeared in~\cite{VB21}.

	\subsubsection{Embedding of partial solutions}
	First, we need the SOPs for $\langle .,.\rangle_{1}$, i.e., the solution of the HIEP formulated for a single Jordan block $Z = J_{n,k_n}\in\mathbb{C}^{(k_n+1)\times(k_n+1)}$ and corresponding weight vector $w=\beta e_{k_n+1}$.
	The solution is
	\begin{equation*}
		H_{k_n+1} = J_{n,k_n}^H \text{ and } Q_{k_n+1} = \begin{bmatrix}
			& & 1\\
			& \rddots\\
			1
		\end{bmatrix}.
	\end{equation*}
	Second, for the HIEP for $\langle .,.\rangle_{\hat{S}}$ we assume, without loss of generality, that we know the solution $\hat{H}\in\mathbb{C}^{(m-k_n-1)\times (m-k_n-1)}$ and $\hat{Q}\in\mathbb{C}^{(m-k_n-1)\times (m-k_n-1)}$.
	If $\hat{Z}$ is a single Jordan block, the solution is obtained as described above. Otherwise we can build the solution by running the procedure described below, starting from a single Jordan block and adding one block at a time.\\
	Third, these two solutions are embedded in the larger matrices
	\begin{equation*}
		\dot{Q} := \begin{bmatrix}
			\hat{Q}\\
			& Q_{k_n+1}
		\end{bmatrix} \text{ and } \dot{H} := \begin{bmatrix}
			\hat{H}\\
			& J_{n,k_n}^H
		\end{bmatrix}.
	\end{equation*}
	The basis matrix is unitary, $\dot{Q}^H \dot{Q} = I$, and the Hessenberg matrix is similar to the given Jordan matrix $\dot{Q}^H{Z} \dot{Q} = \dot{H}$.
	However, the first entries of the eigenvectors are not equal to the given weights $\dot{Q} e_1 \neq \frac{w}{\Vert w\Vert_2}$.
	
	\subsubsection{Introduce weights in basis matrix}
	A plane rotation is applied to $\dot{Q}$, whose action alters the first column such that it equals the normalized weight vector $\frac{w}{\Vert w\Vert_2}$.
	Plane rotations $P_k$ are essentially $2\times2$ unitary matrices with parameters $a,b\in\mathbb{C}$, $\vert a \vert^2 + \vert b \vert^2 = 1$, and in this paper they are of the form
	\begin{equation*}
		P_k := \begin{bmatrix}
			\bar{a} &  & -\bar{b} \\
			& I_{m-k-2} \\
			b  & & a \\
			& & & I_{k}
		\end{bmatrix}\in\mathbb{C}^{m\times m}.
	\end{equation*}
	Take $P_{k_n}$ with parameter $a = \frac{\Vert \hat{w}\Vert_2}{\sqrt{\Vert \hat{w}\Vert_2 + \Vert \beta_n e_{k_n} \Vert_2}}$,
	then the matrix $\dot{Q} P_{k_n}^H$ satisfies $\dot{Q} P_{k_n}^H e_1 = \frac{w}{\Vert w\Vert_2}$ and is unitary, since plane rotations are unitary.
	Because the basis matrix in the Jordan canonical decomposition changes, the recurrence matrix changes accordingly,
	\begin{equation*}
		\dot{Q}^H Z \dot{Q} = \dot{H} \rightarrow P_{k_n}\dot{Q}^H Z \dot{Q} P_{k_n}^H = P_{k_n} \dot{H} P_{k_n}^H.
	\end{equation*}
	The resulting recurrence matrix $P_{k_n} \dot{H} P_{k_n}^H$ is not a Hessenberg matrix.
	An example of this change and of the resulting matrix is given in the following example
	\begin{example} \label{example:change}
		The structure of $P_{k_n} \dot{H} P_{k_n}^H$ for $k_n=2$ is shown in Figure \ref{fig:Hdotchange}, a generic nonzero entry is denoted by $\times$. 
		The matrix $P_{k_n} \dot{H} P_{k_n}^H$ is no longer of Hessenberg structure.
		\begin{figure}[!ht]
			\centering
			{\begin{subfigure}[b]{0.33\textwidth} 
	\scalebox{0.7}{\begin{tikzpicture} 
	\matrix (M) [matrix of nodes,left delimiter={[},right delimiter={]}] 
	{$\times$&$\times$&$\times$&$\times$&$\times$&$\times$&&&\\
		$\times$&$\times$&$\times$&$\times$&$\times$&$\times$&&&\\
		&$\times$&$\times$&$\times$&$\times$&$\times$&&&\\
		&&$\times$&$\times$&$\times$&$\times$&&&\\
		&&&$\times$&$\times$&$\times$&&&\\
		&&&&$\times$&$\times$&&&\\
		&&&&&&$\times$&&\\
		&&&&&&$\times$&$\times$&\\
		&&&&&&&$\times$&$\times$\\
	};
	\node[label=center: {{\Large $\rightarrow$}}]() at (3.5,0){};
\end{tikzpicture} }
	\caption*{$\dot{H}$}
\end{subfigure} 
\hspace{0.7cm}
\begin{subfigure}[b]{0.33\textwidth} 
	\scalebox{0.7}{\begin{tikzpicture} 
 \matrix (M) [matrix of nodes,left delimiter={[},right delimiter={]}] 
{$\times$&$\times$&$\times$&$\times$&$\times$&$\times$&$\times$&&\\
$\times$&$\times$&$\times$&$\times$&$\times$&$\times$&$\times$&&\\
&$\times$&$\times$&$\times$&$\times$&$\times$&&&\\
&&$\times$&$\times$&$\times$&$\times$&&&\\
&&&$\times$&$\times$&$\times$&&&\\
&&&&$\times$&$\times$&&&\\
$\times$&$\times$&$\times$&$\times$&$\times$&$\times$&$\times$&&\\
$\times$&&&&&&$\times$&$\times$&\\
&&&&&&&$\times$&$\times$\\
};
 \end{tikzpicture} }
	\caption*{$P_{k_n}\dot{H}P_{k_n}^H$}
\end{subfigure} }
			\caption{On the left, the structure of $\dot{H}$ and on the right, $P_{k_n} \dot{H} P_{k_n}^H$, the matrix obtained after introducing the new weights as first entries in the last $3$ $(=k_n+1)$ eigenvectors in the basis matrix and applying the corresponding transformation to the recurrence matrix. Generic nonzeros are denoted by $\times$.}
			\label{fig:Hdotchange}
		\end{figure}
	\end{example}

	\subsubsection{Restore the Hessenberg structure}
	Via unitary similarity transformations the matrix $P_{k_n}\dot{H}P_{k_n}^H$ is reduced to Hessenberg form without altering the first column of $\dot{Q} P_{k_n}^H$.
	To this end, any unitary matrix of the form $(1\oplus \breve{Q})
\in\mathbb{C}^{m\times m}$ can be used.
	We use an approach that works column by column, starting at the first column and working up to column $(m-2)$.
	In each column the entries below the first subdiagonal are eliminated.
	This is achieved by a unitary similarity transformation with $\breve{Q}^{\left[i\right]}$.
	We define $\dot{H}^{\left[0\right]} := P_{k_n}\dot{H}P_{k_n}^H$, then by $\dot{H}^{\left[i\right]}:=\left(\breve{Q}^{\left[i\right]}\right)^H\dot{H}^{\left[i-1\right]}\breve{Q}^{\left[i\right]}$ we denote the matrix which is unitarily similar to $\dot{H}^{\left[i-1\right]}$, where the $i$th column of $\dot{H}^{\left[i-1\right]}$ is restored to Hessenberg structure.
	Thus, the product of these matrices forms the unitary matrix $(1\oplus \breve{Q}) = \prod_{i=1}^{m-2}\breve{Q}^{\left[i\right]}$, which restores the Hessenberg structure of $\dot{H}$ without altering its eigenvalues, i.e., the unitary similarity transformation $(1\oplus \breve{Q})^H \dot{H} (1\oplus \breve{Q})$.
	
	We discuss two choices for the matrices $\breve{Q}^{\left[i\right]}$, Householder reflectors and a product of plane rotations.
	The following procedure first appeared in \cite{VB21} for orthogonal polynomials for an inner product without derivatives.
	An example illustrates our column by column approach and afterwards we describe the process in more detail.
	\begin{example}\label{example:process_red}
		For the same example as in Example \ref{example:change}, Figure \ref{fig:Hdots} shows the structure of $\dot{H}^{\left[i\right]}$.
		In each step the entries marked ${\color{red} \star}$ are isolated in a vector $c\in\mathbb{C}^r$.
		This vector is reduced to $\alpha e_1$ by a unitary matrix $\breve{Q}_c \in\mathbb{C}^{r\times r}$, i.e., $\breve{Q}_c c = \alpha e_1$.
		Next, the matrix $\breve{Q}_c$ is embedded into the matrix $\breve{Q}^{\left[i\right]}\in\mathbb{C}^{m\times m}$ such that $\dot{H}^{\left[i+1\right]} = \breve{Q}^{\left[i\right]} \dot{H}^{\left[i\right]} \left(\breve{Q}^{\left[i\right]}\right)^H$ has Hessenberg structure in its first $(i+1)$ columns.
		\begin{figure}[!ht]
			\centering
			{\begin{subfigure}[b]{0.25\textwidth} 
	\scalebox{0.6}{\begin{tikzpicture} 
 \matrix (M) [matrix of nodes,left delimiter={[},right delimiter={]}] 
{$\times$&$\times$&$\times$&$\times$&$\times$&$\times$&&&\\
$\times$&$\times$&$\times$&$\times$&$\times$&$\times$&&&\\
&$\times$&$\times$&$\times$&$\times$&$\times$&&&\\
&&$\times$&$\times$&$\times$&$\times$&&&\\
&&&$\times$&$\times$&$\times$&&&\\
&&&&$\times$&$\times$&&&\\
&&&&&&$\times$&&\\
&&&&&&$\times$&$\times$&\\
&&&&&&&$\times$&$\times$\\
};
 \end{tikzpicture} }
	\caption*{$\dot{H}$}
\end{subfigure} 
\hspace{1cm} 
\begin{subfigure}[b]{0.25\textwidth} 
	\scalebox{0.6}{\begin{tikzpicture} 
	\matrix (M) [matrix of nodes,left delimiter={[},right delimiter={]}] 
	{$\times$&$\times$&$\times$&$\times$&$\times$&$\times$&$\times$&&\\
		{\color{red}$\star$}&$\times$&$\times$&$\times$&$\times$&$\times$&$\times$&&\\
		&$\times$&$\times$&$\times$&$\times$&$\times$&&&\\
		&&$\times$&$\times$&$\times$&$\times$&&&\\
		&&&$\times$&$\times$&$\times$&&&\\
		&&&&$\times$&$\times$&&&\\
		{\color{red}$\star$}&$\times$&$\times$&$\times$&$\times$&$\times$&$\times$&&\\
		{\color{red}$\star$}&&&&&&$\times$&$\times$&\\
		&&&&&&&$\times$&$\times$\\
	};
\end{tikzpicture} }
	\caption*{$\dot{H}^{\left[0\right]}$}
\end{subfigure} 
\hspace{1cm} 
\begin{subfigure}[b]{0.25\textwidth} 
	\scalebox{0.6}{\begin{tikzpicture} 
 \matrix (M) [matrix of nodes,left delimiter={[},right delimiter={]}] 
{$\times$&$\times$&$\times$&$\times$&$\times$&$\times$&$\times$&$\times$&\\
$\times$&$\times$&$\times$&$\times$&$\times$&$\times$&$\times$&$\times$&\\
&{\color{red}$\star$}&$\times$&$\times$&$\times$&$\times$&$\times$&&\\
&&$\times$&$\times$&$\times$&$\times$&&&\\
&&&$\times$&$\times$&$\times$&&&\\
&&&&$\times$&$\times$&&&\\
&{\color{red}$\star$}&$\times$&$\times$&$\times$&$\times$&$\times$&$\times$&\\
&{\color{red}$\star$}&$\times$&$\times$&$\times$&$\times$&$\times$&$\times$&\\
&{\color{red}$\star$}&&&&&$\times$&$\times$&$\times$\\
};
 \end{tikzpicture} }
	\caption*{$\dot{H}^{\left[1\right]}$}
\end{subfigure} 
%%%%%%%%%%%%%%%%%%%%%%%%%%%%%%%% 

\vspace{0.5cm} 

\begin{subfigure}[b]{0.25\textwidth} 
	\scalebox{0.6}{\begin{tikzpicture} 
 \matrix (M) [matrix of nodes,left delimiter={[},right delimiter={]}] 
{$\times$&$\times$&$\times$&$\times$&$\times$&$\times$&$\times$&$\times$&$\times$\\
$\times$&$\times$&$\times$&$\times$&$\times$&$\times$&$\times$&$\times$&$\times$\\
&$\times$&$\times$&$\times$&$\times$&$\times$&$\times$&$\times$&$\times$\\
&&{\color{red}$\star$}&$\times$&$\times$&$\times$&$\times$&&\\
&&&$\times$&$\times$&$\times$&&&\\
&&&&$\times$&$\times$&&&\\
&&{\color{red}$\star$}&$\times$&$\times$&$\times$&$\times$&$\times$&$\times$\\
&&{\color{red}$\star$}&$\times$&$\times$&$\times$&$\times$&$\times$&$\times$\\
&&{\color{red}$\star$}&$\times$&$\times$&$\times$&$\times$&$\times$&$\times$\\
};
 \end{tikzpicture} }
	\caption*{$\dot{H}^{\left[2\right]}$}
\end{subfigure} 
\hspace{1cm} 
\begin{subfigure}[b]{0.25\textwidth} 
	\scalebox{0.6}{\begin{tikzpicture} 
 \matrix (M) [matrix of nodes,left delimiter={[},right delimiter={]}] 
{$\times$&$\times$&$\times$&$\times$&$\times$&$\times$&$\times$&$\times$&$\times$\\
$\times$&$\times$&$\times$&$\times$&$\times$&$\times$&$\times$&$\times$&$\times$\\
&$\times$&$\times$&$\times$&$\times$&$\times$&$\times$&$\times$&$\times$\\
&&$\times$&$\times$&$\times$&$\times$&$\times$&$\times$&$\times$\\
&&&{\color{red}$\star$}&$\times$&$\times$&$\times$&&\\
&&&&$\times$&$\times$&&&\\
&&&{\color{red}$\star$}&$\times$&$\times$&$\times$&$\times$&$\times$\\
&&&{\color{red}$\star$}&$\times$&$\times$&$\times$&$\times$&$\times$\\
&&&{\color{red}$\star$}&$\times$&$\times$&$\times$&$\times$&$\times$\\
};
 \end{tikzpicture} }
	\caption*{$\dot{H}^{\left[3\right]}$}
\end{subfigure} 
\hspace{1cm} 
\begin{subfigure}[b]{0.25\textwidth} 
	\scalebox{0.6}{\begin{tikzpicture} 
 \matrix (M) [matrix of nodes,left delimiter={[},right delimiter={]}] 
{$\times$&$\times$&$\times$&$\times$&$\times$&$\times$&$\times$&$\times$&$\times$\\
$\times$&$\times$&$\times$&$\times$&$\times$&$\times$&$\times$&$\times$&$\times$\\
&$\times$&$\times$&$\times$&$\times$&$\times$&$\times$&$\times$&$\times$\\
&&$\times$&$\times$&$\times$&$\times$&$\times$&$\times$&$\times$\\
&&&$\times$&$\times$&$\times$&$\times$&$\times$&$\times$\\
&&&&{\color{red}$\star$}&$\times$&$\times$&&\\
&&&&{\color{red}$\star$}&$\times$&$\times$&$\times$&$\times$\\
&&&&{\color{red}$\star$}&$\times$&$\times$&$\times$&$\times$\\
&&&&{\color{red}$\star$}&$\times$&$\times$&$\times$&$\times$\\
};
 \end{tikzpicture} }
	\caption*{$\dot{H}^{\left[4\right]}$}
\end{subfigure} 
%%%%%%%%%%%%%%%%%%%%%%%%%%%%%%%% 

\vspace{0.5cm} 

\begin{subfigure}[b]{0.25\textwidth} 
	\scalebox{0.6}{\begin{tikzpicture} 
 \matrix (M) [matrix of nodes,left delimiter={[},right delimiter={]}] 
{$\times$&$\times$&$\times$&$\times$&$\times$&$\times$&$\times$&$\times$&$\times$\\
$\times$&$\times$&$\times$&$\times$&$\times$&$\times$&$\times$&$\times$&$\times$\\
&$\times$&$\times$&$\times$&$\times$&$\times$&$\times$&$\times$&$\times$\\
&&$\times$&$\times$&$\times$&$\times$&$\times$&$\times$&$\times$\\
&&&$\times$&$\times$&$\times$&$\times$&$\times$&$\times$\\
&&&&$\times$&$\times$&$\times$&$\times$&$\times$\\
&&&&&{\color{red}$\star$}&$\times$&$\times$&$\times$\\
&&&&&{\color{red}$\star$}&$\times$&$\times$&$\times$\\
&&&&&{\color{red}$\star$}&$\times$&$\times$&$\times$\\
};
 \end{tikzpicture} }
	\caption*{$\dot{H}^{\left[5\right]}$}
\end{subfigure} 
\hspace{1cm} 
\begin{subfigure}[b]{0.25\textwidth} 
	\scalebox{0.6}{\begin{tikzpicture} 
 \matrix (M) [matrix of nodes,left delimiter={[},right delimiter={]}] 
{$\times$&$\times$&$\times$&$\times$&$\times$&$\times$&$\times$&$\times$&$\times$\\
$\times$&$\times$&$\times$&$\times$&$\times$&$\times$&$\times$&$\times$&$\times$\\
&$\times$&$\times$&$\times$&$\times$&$\times$&$\times$&$\times$&$\times$\\
&&$\times$&$\times$&$\times$&$\times$&$\times$&$\times$&$\times$\\
&&&$\times$&$\times$&$\times$&$\times$&$\times$&$\times$\\
&&&&$\times$&$\times$&$\times$&$\times$&$\times$\\
&&&&&$\times$&$\times$&$\times$&$\times$\\
&&&&&&{\color{red}$\star$}&$\times$&$\times$\\
&&&&&&{\color{red}$\star$}&$\times$&$\times$\\
};
 \end{tikzpicture} }
	\caption*{$\dot{H}^{\left[6\right]}$}
\end{subfigure} 
\hspace{1cm} 
\begin{subfigure}[b]{0.25\textwidth} 
	\scalebox{0.6}{\begin{tikzpicture} 
 \matrix (M) [matrix of nodes,left delimiter={[},right delimiter={]}] 
{$\times$&$\times$&$\times$&$\times$&$\times$&$\times$&$\times$&$\times$&$\times$\\
$\times$&$\times$&$\times$&$\times$&$\times$&$\times$&$\times$&$\times$&$\times$\\
&$\times$&$\times$&$\times$&$\times$&$\times$&$\times$&$\times$&$\times$\\
&&$\times$&$\times$&$\times$&$\times$&$\times$&$\times$&$\times$\\
&&&$\times$&$\times$&$\times$&$\times$&$\times$&$\times$\\
&&&&$\times$&$\times$&$\times$&$\times$&$\times$\\
&&&&&$\times$&$\times$&$\times$&$\times$\\
&&&&&&$\times$&$\times$&$\times$\\
&&&&&&&$\times$&$\times$\\
};
 \end{tikzpicture} }
	\caption*{$\dot{H}^{\left[7\right]}$}
\end{subfigure} }
			\caption{Structure of the recurrence matrix throughout the column by column approach that restores the Hessenberg structure of $\dot{H}^{\left[0\right]}$.
				Entries ${\color{red} \star}$ denote the entries which are used to construct $\breve{Q}^{\left[i\right]}$, the matrix which eliminate the entries below the first subdiagonal of $\dot{H}^{\left[i-1\right]}$ by unitary similarity transformation.
			}
			\label{fig:Hdots}
		\end{figure}
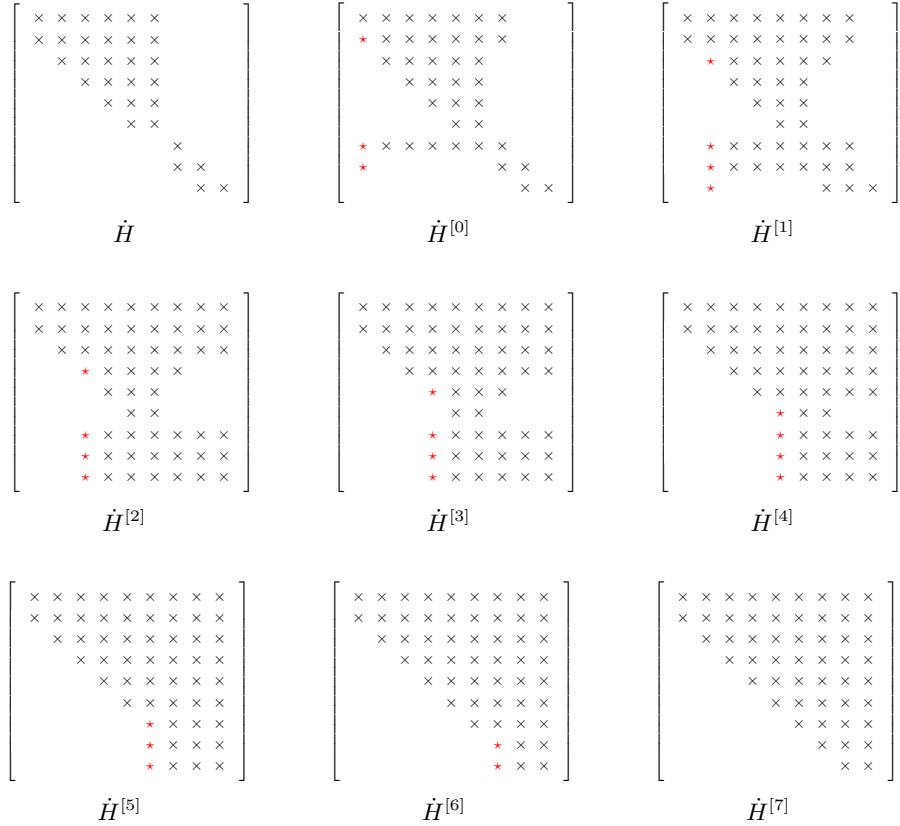
	\end{example}
	We now describe the general process illustrated in the above example, i.e., how to construct matrices $\breve{Q}^{\left[i\right]}$.
	For column $i$ we isolate the relevant entries in a vector $c$, these are the entry on the first subdiagonal $h_{i+1,i}$ and all nonzero entries below this entry, i.e., the entries that are not conform with the Hessenberg structure. 
	If $1\leq i<m-k_n$, set $r := \min\{k_n+1,i+1\}$, and the vector is $c = \begin{bmatrix}
		h_{i+1,i}&
		\dot{h}^{\left[i\right]}_{m-k_n+1,i}&
		\dot{h}^{\left[i\right]}_{m-k_n+2,i}&
		\dots&
		\dot{h}^{\left[i\right]}_{m-k_n+r,i}
	\end{bmatrix}^\top$.
	If $i\geq m-k_n$, set $r:= \min\{k_n+1,m-k_n+i\}$, and the vector is $c = \begin{bmatrix}
		\dot{h}^{\left[i\right]}_{i+1,i}&
		\dot{h}^{\left[i\right]}_{i+2,i}&
		\dots&
		\dot{h}^{\left[i\right]}_{i+r,i}
	\end{bmatrix}^\top$.
	Next, the vector $c$ is reduced to a multiple of $e_1$, this can be done by, e.g., a Householder reflector or a product of plane rotations.
	
	Let $\alpha = \pm \Vert x \Vert_2 e^{\imath \arg(x e_1)}\in\mathbb{C}$, then the Householder reflector \cite{KaSa99}%{\cite[Appendix B.1]{KaSa99}}
	\begin{equation*}
		R = I - 2 \frac{y y^H}{y^H y}, \quad \text{with } y =c + \alpha e_1
	\end{equation*}
	transforms $c$ into a multiple of $e_1$, i.e., $R c = -\alpha e_1$.
	
	When using plane rotations, the entries of $c$ are annihilated one by one.
	The vector $c\in\mathbb{C}^r$ is reduced by the product of plane rotations $\prod_{j=1}^{r-1}P_{j-2}\in\mathbb{C}^{r\times r}$, i.e., $\left(\prod_{j=1}^{r-1}P_{j-2}\right) c = \alpha e_1$, where $\vert\alpha\vert = \Vert c\Vert_2$.
	Note that this product of plane rotations is one of many possible choices.
	
	The matrix thus obtained, $R\in\mathbb{C}^{r\times r}$ for Householder reflectors or $\prod_{j=1}^{r-1}P_{j-2}\in\mathbb{C}^{r\times r}$ for plane rotations, must be embedded in a matrix $\breve{Q}^{\left[i\right]}$ of size $m\times m$ in such a way that it acts on the appropriate rows and column of $\dot{H}^{\left[i-1\right]}$, e.g., in Figure \ref{fig:Hdots} it must act on the entries denoted as {\color{red} $\star$}.
	The unitary similarity transform with these embedded matrices $\left(\breve{Q}^{\left[i\right]}\right)^H \dot{H}^{\left[i-1\right]} \breve{Q}^{\left[i\right]}$ results in $\dot{H}^{\left[i\right]}$ and the above process can be repeated until we reach $\dot{H}^{\left[m-2\right]}$.
	A MATLAB implementation of the updating procedure is available online \cite{git_SOP}.
	
	\section{Numerical experiments}\label{sec:num}
	Numerical experiments are performed which confirm the validity of the procedures proposed in this paper, i.e., the Arnoldi iteration and updating procedure described in Sections \ref{sec:Arnoldi} and \ref{sec:up}, respectively.
	These are also compared to the modified Chebyshev method and discretized Stieltjes procedure proposed by Gautschi and Zhang \cite{GaZh95}. 
	MATLAB implementations are available on the internet \url{http://www.cs.purdue.edu/archives/2002/wxg/codes}, this code computes a sequence of monic polynomials.
	For the updating procedure we report only the results using plane rotations, since the updating procedure using Householder reflectors is very similar for the experiments performed here.
	
	It is important to note that the Stieltjes procedure and Arnoldi iteration require the storage of the whole sequence of SOPs, i.e., the $m\times k$ basis matrix, in order to compute the recurrence matrix.
	The updating procedure only requires the storage of a single vector, the weight vector $w$ of size $m$.

	Section \ref{sec:exp_quant} reports quantitative experiments, we compare our computed results to numerical results or analytical expressions available in the literature.
	For SOPs of high degree, these expressions are challenging to evaluate due to the build up of rounding errors and possible overflow or underflow.
	Therefore, in Section \ref{sec:exp_qual}, we perform qualitative experiments for higher degree SOPS based on the known location of roots of SOPS and the expected behavior of function approximation using SOPs.
	
	\subsection{Quantitative experiments}\label{sec:exp_quant}
	Two quantitative experiments are performed, computing roots of SOPs and generating higher order recurrence coefficients.
	We compare the discrete Stieltjes procedure, Arnoldi iteration and updating procedure.
	
	\subsubsection{Smallest roots of Laguerre-Sobolev polynomials}
	Tables containing the smallest roots of Laguerre-Sobolev polynomials are available in \cite{MaPePi96}.
	We discretize the inner product \eqref{eq:LaguerreInprod} with a $10$-nodes Gauss-Laguerre quadrature rule.
	Using the resulting discretized inner product \eqref{eq:LaguerreInprod_disc} we compute the recurrence matrix $H_{10}$ for the SOPs $\{p_k\}_{k=0}^9$.
	The eigenvalues of the $k\times k$ leading principal submatrix $H_k$ of $H_{10}$ correspond to the roots of $p_k$.
	Table \ref{table:roots_LagSob1} and Table~\ref{table:roots_LagSob2} show the smallest roots of $p_k$ for the parameters $\gamma = 1,\alpha = -1/2$ and $\gamma = 0.2, \alpha = -0.9$, respectively.
	Our computed roots agree with those reported in \cite{MaPePi96}, which are given up to an accuracy of $10^{-8}$.
	
	\begin{table}[!ht]
		\begin{center}
			\caption{The numerically computed smallest roots of the $k$th degree Laguerre-Sobolev polynomials with parameters $\gamma = 1$ and $\alpha = -1/2$.}
			\label{table:roots_LagSob1}
			\begin{tabular}{llll}
				\toprule
				$k$& Stieltjes procedure           & Arnoldi iteration             & Updating procedure            \\ \midrule
				1  & 0.5                 & 0.5                 & 0.5                 \\
				2  & 0.0515973733627622  & 0.0515973733627619  & 0.0515973733627622  \\
				3  & -0.0709467328567679 & -0.0709467328567679 & -0.0709467328567679 \\
				4  & -0.0874916640141531 & -0.0874916640141535 & -0.0874916640141535 \\
				5  & -0.0799899984977785 & -0.0799899984977785 & -0.0799899984977783 \\
				6  & -0.0689833230536414 & -0.0689833230536414 & -0.068983323053641 \\
				7  & -0.0591475889953297 & -0.059147588995331 & -0.0591475889953299 \\
				8  & -0.0512004191713637 & -0.0512004191713639 & -0.0512004191713638 \\
				9  & -0.0449179698365343 & -0.0449179698365336 & -0.0449179698365332 \\
				10 & -0.0399294766750048 & -0.0399294766753265 & -0.0399294766753251
\\
				\bottomrule
			\end{tabular}
		\end{center}
	\end{table}
	
	\begin{table}[!ht]
		\begin{center}
			\caption{The numerically computed smallest roots of the $k$th degree Laguerre-Sobolev polynomials with parameters $\gamma = 0.2$ and $\alpha = -0.9$.}
			\label{table:roots_LagSob2}
			\begin{tabular}{llll}
				\toprule			
				$k$& Stieltjes procedure           & Arnoldi iteration             & Updating procedure            \\ \midrule
				1  & 0.1                 & 0.1               & 0.1                 \\
				2  & -0.0261349584030075  & -0.0261349584030074  & -0.0261349584030074  \\
				3  & -0.0750911669982844 & -0.0750911669982843 & -0.0750911669982843 \\
				4  & -0.0830880010863874 & -0.0830880010863875 & -0.0830880010863876 \\
				5  & -0.0777522363825044 & -0.0777522363825043 & -0.0777522363825047 \\
				6  & -0.0694388792472847 & -0.0694388792472855 & -0.0694388792472857 \\
				7  & -0.0612413492735956 & -0.0612413492735963 & -0.0612413492735955 \\
				8  & -0.0539763658835047 & -0.0539763658835064 & -0.0539763658835068 \\
				9  & -0.0477639920520727 & -0.047763992052076 & -0.0477639920520759 \\
				10 & -0.0425173192185196 & -0.042517319218519 & -0.0425173192185195
\\
				\bottomrule
			\end{tabular}
			
		\end{center}
	\end{table}	
	
	\subsubsection{Higher-order recurrence relations}
	A recent paper \cite{HeHuLaMa22} describes a five-term higher order recurrence relation for SOPs with respect to the discrete Laguerre-Sobolev inner product \eqref{eq:LagSob}.
	The corresponding recurrence matrix is the pentadiagonal matrix $B_m\in \mathbb{C}^{m\times m}$ and can be obtained as follows:
	\begin{enumerate}
		\item Compute an $(m+1)$-point Gauss-Laguerre quadrature rule, with nodes and weights $\{ x_k^{\textrm{La}}\}_{k=1}^{m+1},\{ \vert \beta_k^{\textrm{La}} \vert^2\}_{k=1}^{m+1}$ and construct a discretized inner product of the form \eqref{eq:discreteSobInprod_disc}.
		\item Solve the HIEP for the matrix $(Z-cI)$ and vector $w$, with the matrix and vector as in \eqref{eq:discLagSobMV}. The solution is the Hessenberg matrix $H_{m+1}$.
		\item The pentadiagonal matrix $B_m\in\mathbb{C}^{m\times m}$ is obtained by taking the $m\times m$ leading principal submatrix of $ H_{m+1}^2$.
	\end{enumerate}
	We compare the matrix $\tilde{B}_5$ reported in \cite{HeHuLaMa22} with the coefficients from $B_5$ computed as described above.
	The error is measured in the relative Frobenius norm ${(\Vert B_5-\tilde{B}_5\Vert_\textrm{fro})}/{\Vert \tilde{B}_5 \Vert_\textrm{fro}}$.
	For the Arnoldi iteration we obtain an error of $4.5\mathrm{e}-16$ and for the updating procedure $8.5\mathrm{e}-16$.
	Thus, both methods compute the coefficients up to high accuracy.
	
	On a theoretical note, in the matrix interpretation the five-term recurrence relation might follow from combining the Faber-Manteuffel theorem \cite{FaMa84} with the following equality for the Jordan matrix $((Z-cI)^2)^H = (Z-cI)^2$. Or, it can be studied by splitting $Z$ into the sum of a diagonal Hermitian matrix and a rank one matrix containing only the off-diagonal entry, as in \cite{BaMa00}.
	Neither of these results are directly applicable to this higher order recurrence relation, exploring this further is subject of future research.	
	
	\subsection{Qualitative experiments}\label{sec:exp_qual}
	To test the numerical procedures on higher degree SOPs we perform two qualitative experiments.
	In the first we compute the roots of Althammer polynomials and checks if they are contained in the expected interval.
	In the second we solve a least squares problem that appeared in \cite{IsKoNoSS89}.
	\subsubsection{Roots of Althammer polynomials}
	The roots of Althammer polynomials \cite{Al62} are simple, real and contained in the interval $\left[-1,1\right]$.
	A quadrature rule with $60$ nodes is used to discretize the Sobolev inner product \eqref{eq:AltPoly}.
	Using the four methods we compute the roots of the SOPs of degree $n=50$ and $n=60$ by computing the eigenvalues of the computed Hessenberg recurrence matrices.
	The results are shown in Figure~\ref{fig:n50} and Figure \ref{fig:n60}, respectively.
	For $n=50$ all roots are simple, real and lie in the interval $\left[-1,1\right]$.
	However, for $n=60$ we see that the discretized Stieltjes procedure and modified Chebyshev method provide complex roots.
	This means that the roots obtained by the discretized Stieltjes procedure and modified Chebyshev method are wrong.
	These two methods generate a sequence of monic Sobolev orthogonal polynomials.
	
	We observed that the condition number of the Hessenberg recurrence matrix for monic SOPs is significantly larger, in the order of $10^{16}$, than for the matrix for orthonormal SOPs, order of $10^2$.
	The substantially different results for the computation of the roots can be explained by this difference in normalization and is not due to the numerical algorithms generating the Hessenberg matrix.
	The explanation might be the conditioning of the eigenvalue problem, that must be solved to obtain the roots, for the recurrence matrix representing the orthonormal polynomials compared to the one representing monic polynomials.
	Exploring the conditioning of the eigenvalue problem and a more in-depth comparison of the methods is out of the scope of this paper.
	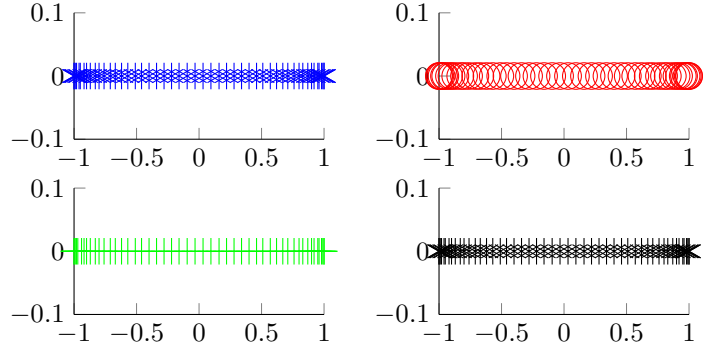
\begin{figure}[!ht]
		\centering
		\setlength\figureheight{4cm}
		\setlength\figurewidth{8cm}
		{% This file was created by matlab2tikz.
%
\begin{tikzpicture}

\begin{axis}[%
width=0.411\figurewidth,
height=0.419\figureheight,
at={(0\figurewidth,0.581\figureheight)},
scale only axis,
xmin=-1,
xmax=1,
ymin=-0.1,
ymax=0.1,
ytick={-0.1,0,0.1},
axis background/.style={fill=white},
axis x line*=bottom,
axis y line*=left
]
\addplot [color=blue, draw=none, mark size=5.0pt, mark=asterisk, mark options={solid, blue}, forget plot]
  table[row sep=crcr]{%
-1	0\\
-1	0\\
-0.99	0\\
-0.98	0\\
-0.96	0\\
-0.95	0\\
-0.92	0\\
-0.9	0\\
-0.87	0\\
-0.83	0\\
-0.8	0\\
-0.76	0\\
-0.71	0\\
-0.67	0\\
-0.62	0\\
-0.57	0\\
-0.51	0\\
-0.46	0\\
-0.4	0\\
-0.34	0\\
-0.28	0\\
-0.22	0\\
-0.16	0\\
-0.095	0\\
-0.032	0\\
0.032	0\\
0.095	0\\
0.16	0\\
0.22	0\\
0.28	0\\
0.34	0\\
0.4	0\\
0.46	0\\
0.51	0\\
0.57	0\\
0.62	0\\
0.67	0\\
0.71	0\\
0.76	0\\
0.8	0\\
0.83	0\\
0.87	0\\
0.9	0\\
0.92	0\\
0.95	0\\
0.96	0\\
0.98	0\\
0.99	0\\
1	0\\
1	0\\
};
\end{axis}

\begin{axis}[%
width=0.411\figurewidth,
height=0.419\figureheight,
at={(0.6\figurewidth,0.581\figureheight)},
scale only axis,
xmin=-1,
xmax=1,
ymin=-0.1,
ymax=0.1,
ytick={-0.1,0,0.1},
axis background/.style={fill=white},
axis x line*=bottom,
axis y line*=left
]
\addplot [color=red, draw=none, mark size=5.0pt, mark=o, mark options={solid, red}, forget plot]
  table[row sep=crcr]{%
1	0\\
1	0\\
0.99	0\\
0.98	0\\
0.96	0\\
0.95	0\\
0.92	0\\
0.9	0\\
0.87	0\\
0.83	0\\
0.8	0\\
0.76	0\\
0.71	0\\
0.67	0\\
0.62	0\\
0.57	0\\
0.51	0\\
0.46	0\\
0.4	0\\
0.34	0\\
0.28	0\\
0.22	0\\
0.16	0\\
0.095	0\\
0.032	0\\
-0.032	0\\
-0.095	0\\
-0.16	0\\
-0.22	0\\
-0.28	0\\
-0.34	0\\
-1	0\\
-1	0\\
-0.99	0\\
-0.98	0\\
-0.96	0\\
-0.95	0\\
-0.92	0\\
-0.9	0\\
-0.87	0\\
-0.83	0\\
-0.8	0\\
-0.76	0\\
-0.71	0\\
-0.67	0\\
-0.4	0\\
-0.62	0\\
-0.57	0\\
-0.46	0\\
-0.51	0\\
};
\end{axis}

\begin{axis}[%
width=0.411\figurewidth,
height=0.419\figureheight,
at={(0\figurewidth,0\figureheight)},
scale only axis,
xmin=-1,
xmax=1,
ymin=-0.1,
ymax=0.1,
ytick={-0.1,0,0.1},
axis background/.style={fill=white},
axis x line*=bottom,
axis y line*=left
]
\addplot [color=green, draw=none, mark size=5.0pt, mark=+, mark options={solid, green}, forget plot]
  table[row sep=crcr]{%
-1	0\\
-1	0\\
-0.99	0\\
-0.98	0\\
-0.97	0\\
-0.94	0\\
-0.92	0\\
-0.9	0\\
-0.87	0\\
-0.83	0\\
-0.8	0\\
-0.76	0\\
-0.71	0\\
-0.67	0\\
-0.62	0\\
-0.57	0\\
-0.51	0\\
-0.46	0\\
-0.4	0\\
-0.34	0\\
-0.28	0\\
-0.22	0\\
-0.16	0\\
-0.095	0\\
-0.032	0\\
0.032	0\\
0.095	0\\
0.16	0\\
0.22	0\\
0.28	0\\
0.34	0\\
0.4	0\\
0.46	0\\
0.51	0\\
0.57	0\\
0.62	0\\
0.67	0\\
0.71	0\\
0.76	0\\
0.8	0\\
0.83	0\\
0.87	0\\
0.9	0\\
0.92	0\\
0.95	0\\
0.96	0\\
0.98	0\\
0.99	0\\
1	0\\
1	0\\
};
\end{axis}

\begin{axis}[%
width=0.411\figurewidth,
height=0.419\figureheight,
at={(0.6\figurewidth,0\figureheight)},
scale only axis,
xmin=-1,
xmax=1,
ymin=-0.1,
ymax=0.1,
ytick={-0.1,0,0.1},
axis background/.style={fill=white},
axis x line*=bottom,
axis y line*=left
]
\addplot [color=black, draw=none, mark size=5.0pt, mark=asterisk, mark options={solid, black}, forget plot]
  table[row sep=crcr]{%
-1	0\\
-1	0\\
-0.99	0\\
-0.98	0\\
-0.96	0\\
-0.95	0\\
-0.92	0\\
-0.9	0\\
-0.87	0\\
-0.83	0\\
-0.8	0\\
-0.76	0\\
-0.71	0\\
-0.67	0\\
-0.62	0\\
-0.57	0\\
-0.51	0\\
-0.46	0\\
-0.4	0\\
-0.34	0\\
-0.28	0\\
-0.22	0\\
-0.16	0\\
-0.095	0\\
-0.032	0\\
0.032	0\\
0.095	0\\
0.16	0\\
0.22	0\\
0.28	0\\
0.34	0\\
1	0\\
1	0\\
0.99	0\\
0.98	0\\
0.96	0\\
0.95	0\\
0.92	0\\
0.9	0\\
0.87	0\\
0.83	0\\
0.8	0\\
0.76	0\\
0.71	0\\
0.67	0\\
0.4	0\\
0.62	0\\
0.57	0\\
0.46	0\\
0.51	0\\
};
\end{axis}
\end{tikzpicture}%
}	
		\caption{Roots of Althammer polynomials of degree $n=50$ and $\gamma=100$ computed by, from left to right and top to bottom, Stieltjes procedure ${\color{blue} \ast}$, Arnoldi iteration ${\color{red}\circ}$, modified Chebyshev method ${\color{green} +}$ and updating procedure ${\color{black}\ast}$.}
		\label{fig:n50}
	\end{figure}
	
	\begin{figure}[!ht]
		\centering
		\setlength\figureheight{4cm}
		\setlength\figurewidth{8cm}
		{% This file was created by matlab2tikz.
%
\begin{tikzpicture}

\begin{axis}[%
width=0.411\figurewidth,
height=0.419\figureheight,
at={(0\figurewidth,0.581\figureheight)},
scale only axis,
xmin=-1,
xmax=1,
ymin=-0.1,
ymax=0.1,
ytick={-0.1,0,0.1},
axis background/.style={fill=white},
axis x line*=bottom,
axis y line*=left
]
\addplot [color=blue, draw=none, mark size=5.0pt, mark=asterisk, mark options={solid, blue}, forget plot]
  table[row sep=crcr]{%
0.021	0\\
-0.046	-0.012\\
-0.046	0.012\\
0.1	-0.017\\
0.1	0.017\\
-0.15	-0.035\\
-0.15	0.035\\
0.21	-0.019\\
0.21	0.019\\
-0.26	-0.041\\
-0.26	0.041\\
0.3	0\\
0.33	0\\
-0.36	-0.041\\
-0.36	0.041\\
0.38	0\\
0.46	-0.0085\\
0.46	0.0085\\
-0.46	-0.034\\
-0.46	0.034\\
-0.53	0\\
0.55	0\\
0.56	0\\
-0.58	-0.029\\
-0.58	0.029\\
0.61	0\\
0.67	-0.019\\
0.67	0.019\\
-0.67	-0.043\\
-0.67	0.043\\
0.75	-0.027\\
0.75	0.027\\
-0.75	-0.046\\
-0.75	0.046\\
-0.81	-0.045\\
-0.81	0.045\\
0.81	-0.026\\
0.81	0.026\\
-0.87	-0.041\\
-0.87	0.041\\
0.88	-0.017\\
0.88	0.017\\
0.89	0\\
-0.92	-0.035\\
-0.92	0.035\\
0.93	-0.022\\
0.93	0.022\\
-0.96	-0.028\\
-0.96	0.028\\
0.97	-0.02\\
0.97	0.02\\
-0.98	-0.02\\
-0.98	0.02\\
0.99	-0.014\\
0.99	0.014\\
-1	-0.01\\
-1	0.01\\
1	-0.0049\\
1	0.0049\\
-1	0\\
};
\end{axis}

\begin{axis}[%
width=0.411\figurewidth,
height=0.419\figureheight,
at={(0.6\figurewidth,0.581\figureheight)},
scale only axis,
xmin=-1,
xmax=1,
ymin=-0.1,
ymax=0.1,
ytick={-0.1,0,0.1},
axis background/.style={fill=white},
axis x line*=bottom,
axis y line*=left
]
\addplot [color=red, draw=none, mark size=5.0pt, mark=o, mark options={solid, red}, forget plot]
  table[row sep=crcr]{%
1	0\\
1	0\\
0.99	0\\
0.99	0\\
0.98	0\\
0.96	0\\
0.95	0\\
0.93	0\\
0.91	0\\
0.88	0\\
0.86	0\\
0.83	0\\
0.8	0\\
0.77	0\\
0.73	0\\
0.69	0\\
0.65	0\\
0.61	0\\
0.57	0\\
0.53	0\\
0.48	0\\
0.43	0\\
0.39	0\\
0.34	0\\
0.29	0\\
0.24	0\\
0.18	0\\
0.13	0\\
0.079	0\\
0.026	0\\
-0.026	0\\
-0.079	0\\
-0.13	0\\
-0.18	0\\
-0.24	0\\
-0.29	0\\
-0.34	0\\
-1	0\\
-1	0\\
-0.99	0\\
-0.99	0\\
-0.98	0\\
-0.96	0\\
-0.95	0\\
-0.93	0\\
-0.91	0\\
-0.88	0\\
-0.86	0\\
-0.83	0\\
-0.8	0\\
-0.77	0\\
-0.73	0\\
-0.69	0\\
-0.65	0\\
-0.39	0\\
-0.61	0\\
-0.57	0\\
-0.43	0\\
-0.53	0\\
-0.48	0\\
};
\end{axis}

\begin{axis}[%
width=0.411\figurewidth,
height=0.419\figureheight,
at={(0\figurewidth,0\figureheight)},
scale only axis,
xmin=-1,
xmax=1,
ymin=-0.1,
ymax=0.1,
ytick={-0.1,0,0.1},
axis background/.style={fill=white},
axis x line*=bottom,
axis y line*=left
]
\addplot [color=green, draw=none, mark size=5.0pt, mark=+, mark options={solid, green}, forget plot]
  table[row sep=crcr]{%
-0.017	0\\
0.045	-0.0089\\
0.045	0.0089\\
-0.1	-0.023\\
-0.1	0.023\\
0.15	-0.035\\
0.15	0.035\\
-0.21	-0.028\\
-0.21	0.028\\
0.26	-0.041\\
0.26	0.041\\
-0.31	-0.02\\
-0.31	0.02\\
0.36	-0.041\\
0.36	0.041\\
-0.39	0\\
-0.44	0\\
0.46	-0.034\\
0.46	0.034\\
-0.47	0\\
0.53	0\\
-0.55	-0.018\\
-0.55	0.018\\
0.58	-0.029\\
0.58	0.029\\
-0.64	-0.0022\\
-0.64	0.0022\\
0.67	-0.043\\
0.67	0.043\\
-0.68	0\\
-0.74	-0.03\\
-0.74	0.03\\
0.75	-0.046\\
0.75	0.046\\
-0.81	-0.038\\
-0.81	0.038\\
0.81	-0.045\\
0.81	0.045\\
-0.87	-0.038\\
-0.87	0.038\\
0.87	-0.041\\
0.87	0.041\\
-0.92	-0.035\\
-0.92	0.035\\
0.92	-0.035\\
0.92	0.035\\
-0.96	-0.029\\
-0.96	0.029\\
0.96	-0.028\\
0.96	0.028\\
0.98	-0.02\\
0.98	0.02\\
-0.98	-0.021\\
-0.98	0.021\\
1	-0.01\\
1	0.01\\
-1	-0.011\\
-1	0.011\\
1	0\\
-1	0\\
};
\end{axis}

\begin{axis}[%
width=0.411\figurewidth,
height=0.419\figureheight,
at={(0.6\figurewidth,0\figureheight)},
scale only axis,
xmin=-1,
xmax=1,
ymin=-0.1,
ymax=0.1,
ytick={-0.1,0,0.1},
axis background/.style={fill=white},
axis x line*=bottom,
axis y line*=left
]
\addplot [color=black, draw=none, mark size=5.0pt, mark=asterisk, mark options={solid, black}, forget plot]
  table[row sep=crcr]{%
1	0\\
1	0\\
0.99	0\\
0.99	0\\
0.98	0\\
0.96	0\\
0.95	0\\
0.93	0\\
0.91	0\\
0.88	0\\
0.86	0\\
0.83	0\\
0.8	0\\
0.77	0\\
0.73	0\\
0.69	0\\
0.65	0\\
0.61	0\\
0.57	0\\
0.53	0\\
0.48	0\\
0.43	0\\
0.39	0\\
0.34	0\\
0.29	0\\
0.24	0\\
0.18	0\\
0.13	0\\
0.079	0\\
0.026	0\\
-0.026	0\\
-0.079	0\\
-0.13	0\\
-0.18	0\\
-0.24	0\\
-0.29	0\\
-0.34	0\\
-1	0\\
-1	0\\
-0.99	0\\
-0.99	0\\
-0.98	0\\
-0.96	0\\
-0.95	0\\
-0.93	0\\
-0.91	0\\
-0.88	0\\
-0.86	0\\
-0.83	0\\
-0.8	0\\
-0.77	0\\
-0.73	0\\
-0.69	0\\
-0.65	0\\
-0.39	0\\
-0.61	0\\
-0.57	0\\
-0.43	0\\
-0.53	0\\
-0.48	0\\
};
\end{axis}
\end{tikzpicture}%
}				
		\caption{Roots of Althammer polynomials of degree $n=60$ and $\gamma=100$ computed, from left to right and top to bottom, Stieltjes procedure ${\color{blue} \ast}$, Arnoldi iteration ${\color{red}\circ}$, modified Chebyshev method ${\color{green} +}$ and updating procedure ${\color{black}\ast}$.}
		\label{fig:n60}
	\end{figure}
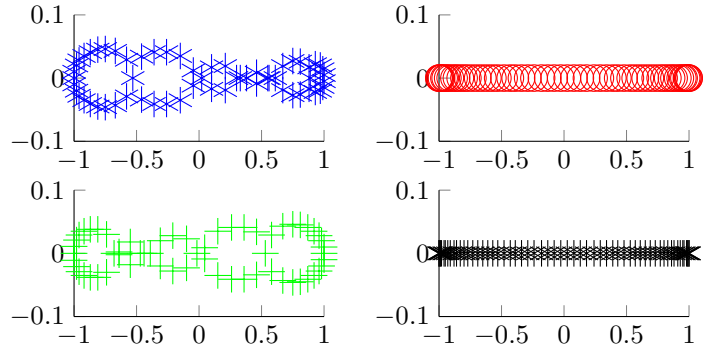
	
	\subsubsection{Least squares problem}
	In least squares function approximation we seek an approximation to an unknown function $f(x)$ which is only known in a set of $m$ given nodes.
	Let the unknown function be, as in \cite{IsKoNoSS89},
	\begin{equation*}
		f(x) = \exp(-100(x-1/5)^2).
	\end{equation*}
	Consider a Hermite least squares problem, where the information available are the function values $\{f(x_j)\}_{j=1}^m$ and the values of its first derivative $\{f'(x_j)\}_{j=1}^m$.
	That is, find the polynomial $f_n$ of degree $n$ that minimizes the least squares criterion
	\begin{equation*}
		\sum_{j=1}^{m} \vert \beta_j \vert^2 \vert f_n(x_j) - f(x_j)\vert^2 + \gamma \sum_{j=1}^{m} \vert \beta_j \vert^2 \vert f'_n(x_j) - f'(x_j)\vert^2.
	\end{equation*}
	A basis of Sobolev orthogonal polynomials is the natural choice of basis in which to formulate the linear system associated to this problem and in which to represent the polynomial approximant $f_n$ \cite{Fo57}.
	Suppose the nodes and weights come from the Gauss-Legendre quadrature rule, that is, $x_j = x_j^{\textrm{Le}}$ and $\beta_j = \beta_j^{\textrm{Le}}$ for $j=1,2,\dots,m$.
	
	For $m=201$ we construct, using the Arnoldi iteration, the least squares polynomial approximant of degree $n=1,11,\dots,201$, for $\gamma = 0$, Legendre polynomials, and $\gamma = 1/100$, Althammer polynomials.
	The error measure in the maximum norm is shown in Figure \ref{fig:LS}, the most accurate approximation to the derivative $f'(x)$ is obtained by the Althammer polynomials. 
	In terms of function values they perform equally well.
	These results correspond to those reported in \cite{IsKoNoSS89}, indicating that the procedure proposed in this paper is valid.
	
	\begin{figure}[!ht]
		\centering
		\setlength\figureheight{2.5cm}
		\setlength\figurewidth{9cm}
		{% This file was created by matlab2tikz.
%
\begin{tikzpicture}

\begin{axis}[%
width=0.411\figurewidth,
height=\figureheight,
at={(0\figurewidth,0\figureheight)},
scale only axis,
xmin=0,
xmax=201,
xlabel style={font=\color{white!15!black}},
xlabel={$n$},
ymode=log,
ymin=1e-16,
ymax=1,
yminorticks=true,
ylabel style={font=\color{white!15!black}},
ylabel={$\Vert f_n(x)-f(x) \Vert_\infty$},
axis background/.style={fill=white}
]
\addplot [color=blue, draw=none, mark=asterisk, mark options={solid, blue}, forget plot]
table[row sep=crcr]{%
	1	0.9110\\
	11	0.469\\
	21	0.149\\
	31	0.0279\\
	41	0.00448\\
	51	0.000324\\
	61	2.04e-05\\
	71	7.79e-07\\
	81	1.74e-08\\
	91	3.66e-10\\
	101	3.51e-12\\
	111	3.89e-14\\
	121	1.22e-15\\
	131	1.11e-15\\
	141	1.11e-15\\
	151	1.11e-15\\
	161	1.11e-15\\
	171	1.11e-15\\
	181	1.11e-15\\
	191	1.11e-15\\
	201	1.11e-15\\
};
\addplot [color=red, draw=none, mark=o, mark options={solid, red}, forget plot]
table[row sep=crcr]{%
	1	0.911\\
	11	0.432\\
	21	0.135\\
	31	0.0364\\
	41	0.00857\\
	51	0.000807\\
	61	6.84e-05\\
	71	2.59e-06\\
	81	7.6e-08\\
	91	1.45e-09\\
	101	1.86e-11\\
	111	1.88e-13\\
	121	1.33e-15\\
	131	1.22e-15\\
	141	1.22e-15\\
	151	1.22e-15\\
	161	1.22e-15\\
	171	1.22e-15\\
	181	1.22e-15\\
	191	1.22e-15\\
	201	1.11e-15\\
};
\end{axis}

\begin{axis}[%
width=0.411\figurewidth,
height=\figureheight,
at={(0.65\figurewidth,0\figureheight)},
scale only axis,
xmin=0,
xmax=201,
xlabel style={font=\color{white!15!black}},
xlabel={$n$},
ymode=log,
ymin=1e-15,
ymax=3.29,
yminorticks=true,
ylabel style={font=\color{white!15!black}},
ylabel={$\Vert f'_n(x)-f'(x) \Vert_\infty$},
axis background/.style={fill=white}
]
\addplot [color=blue, draw=none, mark=asterisk, mark options={solid, blue}, forget plot]
table[row sep=crcr]{%
	1	1\\
	11	0.886\\
	21	0.45\\
	31	0.181\\
	41	0.0458\\
	51	0.00573\\
	61	0.000643\\
	71	2.86e-05\\
	81	1.04e-06\\
	91	2.18e-08\\
	101	3.34e-10\\
	111	3.56e-12\\
	121	2.75e-14\\
	131	1.66e-15\\
	141	1.67e-15\\
	151	2.08e-15\\
	161	2.41e-15\\
	171	2.62e-15\\
	181	2.36e-15\\
	191	2.23e-15\\
	201	3.39e-15\\
};
\addplot [color=red, draw=none, mark=o, mark options={solid, red}, forget plot]
table[row sep=crcr]{%
	1	1\\
	11	1.87\\
	21	3.29\\
	31	2.13\\
	41	0.833\\
	51	0.123\\
	61	0.015\\
	71	0.000762\\
	81	2.95e-05\\
	91	7e-07\\
	101	1.12e-08\\
	111	1.35e-10\\
	121	8.87e-13\\
	131	3.21e-13\\
	141	6.22e-13\\
	151	9.09e-13\\
	161	1.33e-12\\
	171	1.14e-12\\
	181	7.72e-13\\
	191	7.74e-13\\
	201	3.09e-13\\
};
\end{axis}
\end{tikzpicture}%
}		
		\caption{Approximation quality of the least squares solution $f_n(x)$ for the Legendre inner product ${\color{red} \circ}$, $\gamma=0$, and the Althammer inner product ${\color{blue} \ast}$, with $\gamma=1/10$. The error, measured in the maximum norm, for the function itself on the left and for the first derivative on the right.}
		\label{fig:LS}
	\end{figure}
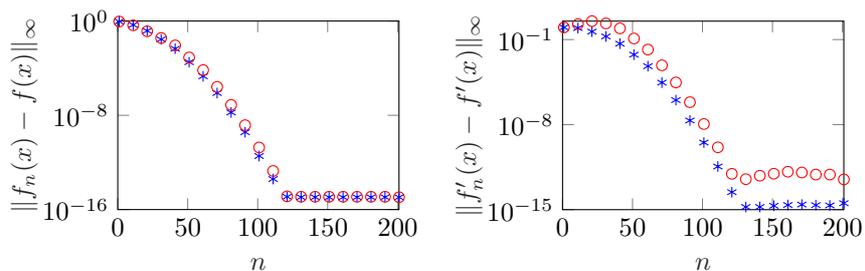
	
	\section{Conclusion and future work}
	A connection between Sobolev orthonormal polynomials and orthonormal bases for Krylov subspaces generated with a Jordan matrix is proved.
	Using this connection the problem of generating Sobolev orthogonal polynomials is reformulated as a Hessenberg inverse eigenvalue problem.
	Such a problem starts from known spectral information, the Jordan matrix containing the eigenvalues of the Hessenberg matrix and the first entries of its eigenvectors, and reconstructs the Hessenberg matrix.
	The reconstructed Hessenberg matrix represents the recurrence relation for Sobolev orthogonal polynomials.
	In the context of Sobolev orthogonal polynomials the spectral information for this inverse eigenvalue problem is obtained by discretizing the associated Sobolev inner product with a Gauss quadrature rule.
	Two new numerical procedures to solve the Hessenberg inverse eigenvalue problem are proposed.
	Numerical experiments show that they are competitive with the state-of-the-art methods, they produce results of similar accuracy.
	The updating procedure is the most efficient in terms of memory.
	A more thorough comparison between these methods, an analysis of their computational cost and potential parallelization of the updating procedure are subject of future research.
	
	This new connection between Sobolev orthonormal polynomials and Krylov subspaces opens the possibility to use efficient and robust numerical algorithms from the field of numerical linear algebra for computing (with) Sobolev orthogonal polynomials.
	The development of spectral solvers for differential equations based on these polynomials \cite{YuWaLi19} has received some attention in recent years, efficient numerical algorithms will be paramount to their success.
	
	\section*{Acknowledgements}
	The author would like to thank Francisco Marcell\'an and Stefano Pozza for comments on an earlier draft of this paper which improved its presentation greatly and Petr Tich\'y for suggesting valuable references. 

	\section*{Funding}
	The research of the author was supported by Charles University Research program No. PRIMUS/21/SCI/009.
	
	\bibliographystyle{siam}
	\bibliography{/home/buggenhout/Research/references}  % name your BibTeX data base

	\end{document}